\documentclass{amsart}

\usepackage{amssymb}
\usepackage{amsmath}
\usepackage{enumitem}
\usepackage{amsthm}
\usepackage{amsbsy}
\usepackage{bm}
\usepackage{hyperref,cleveref}
\usepackage{caption,subcaption}
\usepackage{mathrsfs}
\usepackage{pinlabel}
\usepackage{thmtools, thm-restate}

\theoremstyle{plain}
\newtheorem{theorem}{Theorem}[section]
\newtheorem{question}[theorem]{Question}
\newtheorem{lemma}[theorem]{Lemma}

\newtheorem{proposition}[theorem]{Proposition}

\theoremstyle{definition}
\newtheorem{definition}[theorem]{Definition}

\theoremstyle{remark}
\newtheorem{remark}[theorem]{Remark}
\newtheorem{example}[theorem]{Example}

\newcommand{\tra}{\mathrm{tr}}
\newcommand{\pslr}{\mathrm{PSL}(2,\mathbb{R})}
\newcommand{\pslc}{\mathrm{PSL}(2,\mathbb{C})}
\newcommand{\slc}{\mathrm{SL}(2,\mathbb{C})}
\newcommand{\slr}{\mathrm{SL}(2,\mathbb{R})}

\usepackage{tikz-cd}
\usepackage{xparse}
\usepackage[most]{tcolorbox}
\usepackage{amsthm}
\usepackage{todonotes}

\usepackage{mathtools}
\usepackage{adjustbox}

\begin{document}

\title[On dynamics of the $\text{MCG}(S_{g,n})$ on $\pslr$ Characters]{On dynamics of the Mapping class group action on relative $\pslr$-Character Varieties}
\author[A.K.Nair]{Ajay Kumar Nair}
\email{ajaynair@iisc.ac.in}

\address{	Department of Mathematics\\
  Indian Institute of Science\\
  Bangalore 560012, India}

\begin{abstract} 
In this paper, we study the mapping class group action on the relative $\pslr$-character varieties of punctured surfaces. It is well known that Minsky's primitive-stable representations form a domain of discontinuity for the $\text{Out}(F_n)$-action on the $\pslc$-character variety. We define simple-stability of representations of fundamental group of a surface into $\pslr$ which is an analogue of the definition of primitive stability and prove that these representations form a domain of discontinuity for the $\text{MCG}$-action. Our first main result shows that holonomies of hyperbolic cone surfaces are simple-stable. We also prove that holonomies of hyperbolic cone surfaces with exactly one cone-point of cone-angle less than $\pi$ are primitive-stable, thus giving examples of an infinite family of indiscrete primitive-stable representations.
\end{abstract}
\maketitle

\section{Introduction}

In this paper we are concerned with representations of the fundamental group of punctured surfaces into $\pslr$. We will first motivate the results obtained in the paper by surveying some results about representations of fundamental groups of closed surfaces. Let $S_g$ represent a closed surface of genus $g$. For now, we focus on $g \geq 2$. 
In \cite{GOLDMAN2}, Goldman demonstrated that the $\pslr$-character variety, has $4g - 3$ connected components, which are indexed by the Euler numbers of representations. Of particular interest are the components with maximal Euler numbers, $\pm (2g - 2)$, which correspond to the Teichmüller space of $S_g$, or the space of marked hyperbolic structures on $S_g$. These components also coincide with the space of discrete and faithful representations. 

The mapping class group of $S_g$ (see Definition~\ref{def:MCG}) acts on this character variety, preserving the connected components described above. A classical result (see \cite[Theorem 12.2]{FB}) establishes that this action is properly discontinuous on the Teichmüller space. Goldman conjectured that the mapping class group action is ergodic with respect to the symplectic measure on non-maximal components of the character variety (see \cite[Conjecture 3.1]{GOLDMANConj}). In  \cite{Marche-Wolff}, Marché and Wolff confirmed this conjecture in the case $g = 2$ for Euler number $\pm 1$ but showed that it does not hold for the component with Euler number $0$.

Now, we will turn to our case i.e., the punctured surfaces. Let $S_{g,n}$ denote a surface with genus $g$ and $n$ punctures, where $n \geq 1$. The fundamental group, $\pi_1(S_{g,n})$, of $S_{g,n}$  is given by $F_{2g+ n -1}$ i.e., free group on $2g + n - 1$ generators. We will be interested in $\pslr$-character varieties of fundamental groups of punctured surfaces. The group $\pslr$, as the group of orientation preserving isometries of $\mathbb{H}^2$, is particularly interesting because discrete and faithful representations from fundamental groups of surfaces into $\pslr$ naturally arise as holonomies of hyperbolic structures on surfaces. Similarly, $\pslc$-character varieties are also widely studied since $\pslc$ is the group of orientation preserving isometries of $\mathbb{H}^3$. Again, the discrete and faithful representations here arise as holonomies of hyperbolic structures on 3-manifolds.

 The group of outer automorphisms of the free group $F_n$, denoted by $\text{Out}(F_n)$, acts by precomposition on the character varieties. In \cite{Minsky}, Minsky defined primitive-stable representations of the free group $F_n$ into $\pslc$ and proved that they form a domain of discontinuity for the $\text{Out}(F_n)$-action on $\pslc$-character variety. Moreover, this domain of discontinuity is strictly larger than the Schottky representations and contains indiscrete representations. 

Interpreting $F_n$ as a fundamental group of a punctured surface, the mapping class group $\text{MCG}(S_{g,n})$ is isomorphic to a subgroup of $\text{Out}(F_n)$ \cite[Theorem 8.8]{FB} and the analogue of cyclically reduced primitive elements are the free homotopy classes of simple closed curves on the surface. Also, this action preserves the conjugacy class of peripheral simple closed curves, giving a well-defined action on relative character varieties, i.e., a subspace of character variety with fixed peripheral conjugacy classes (see Section~\ref{sec:CharacterVariety}). 

In this context, we define the notion of simple-stability (see Definition~\ref{def:SimpleStability}), for punctured surfaces of genus at least one. Note that a similar definition for closed surfaces can be found in \cite[Section 5]{Minsky}. This becomes an obvious counterpart for primitive-stability in the context of $\text{MCG}(S_{g,n})$-action on relative $\pslr$-character varieties. Let $\mathcal{SS}(S_{g,n})$ denote the space of simple-stable representations and $\mathcal{X}(S_{g,n}, \mathcal{C})$  denote the relative $\pslr$-character variety with $\mathcal{C}$ a fixed set of $n$ conjugacy classes (see Section~\ref{sec:CharacterVariety}). Following \cite{Minsky}, we prove: 

\begin{restatable}{theorem}{ProperDiscontinuity}\label{thm:propdiscontinuity}
	$\textup{MCG}(S_{g,n})$ acts properly discontinuously on $\mathcal{SS}(S_{g,n})$. In particular, $\textup{MCG}(S_{g,n})$ acts properly discontinuously on $\mathcal{SS}(S_{g,n}) \cap \mathcal{X}(S_{g,n}, \mathcal{C})$, where $\mathcal{C}$ is a set of $n$ conjugacy classes in $\pslr$.
\end{restatable}

We will now look at elements in relative character varieties that arise as the holonomies of hyperbolic cone surfaces. For definition and details on hyperbolic cone surfaces, see Section~\ref{sec:HyperbolicConeSurfaces}. For other relevant results on hyperbolic cone surfaces, see \cite{HuipingPan, PARLIER, SPTAN}. The holonomies of hyperbolic cone surfaces are indiscrete if one of the cone-angles is irrational. Thus, these form examples of \textit{geometrisable} indiscrete representations i.e., indiscrete representations having an equivariant immersion from the universal cover of the surface into $\mathbb{H}^2$. We prove the following theorems for these holonomies:

\begin{restatable}{theorem}{ConeSurfacesSimpleStability}\label{thm:main1}
    The holonomies of admissible cone surfaces are simple-stable. Furthermore, in the special case where there is exactly one cone-point, the holonomies are primitive-stable.
\end{restatable}

The first part of the above theorem gives us examples of simple-stable representations that are not Schottky and the second part gives us examples of indiscrete primitive-stable representations answering a question of Minsky in \cite[Section 5, p.66]{Minsky}.

The case when the surface is a one-holed torus has been explored by many authors, see \cite{GOLDMAN, Lee-Xu, TWZ1}. In \cite{GOLDMAN}, Goldman completely describes the mapping class group action on the relative $\slr$-character varieties of one-holed torus. In particular, Goldman proves that the $\text{Out}(F_2)$ acts properly on the space of holonomies of hyperbolic cone torus with a cone-point of fixed cone-angle $\theta < 2 \pi$ \cite[Section 3.5]{GOLDMAN}. It is interesting to note that in the case of one-holed torus, $\text{MCG}(S_{1,1}) = \text{Out}(F_2)$ and primitive elements are exactly the free homotopy classes of simple closed curves (except for the boundary curve). Theorem~\ref{thm:main1} together with Theorem~\ref{thm:propdiscontinuity} can be seen as generalisation of this result of Goldman. 

In \cite{Bowditch1}, for representations of $F_2$ into $\pslc$, Bowditch defined the notion of $BQ$-conditions and in \cite{TWZ1}, Tan-Wong-Zhang proved many results for $BQ$- representations. One of the important results they prove is that the representations satisfying the $BQ$-conditions form an open subset of the character variety and $\text{Out}(F_2)$ acts properly discontinuously on this subset. In \cite{Bowditch1}, Bowditch conjectures that a type-preserving representation (i.e., non-elementary representation with peripheral element going to a parabolic) satisfies $BQ$-conditions iff it is quasifuchsian (see Conjecture A in \cite{Bowditch1}). This conjecture still remains open.

Denote the primitive-stable representations of $F_n$ by $\mathcal{PS}(F_n)$ and the representations satisfying $BQ$-conditions by $\mathcal{BQ}(F_n)$. For more on these definitions see Sections~\ref{sec:SimpleStability}, \ref{sec:SBQconditions}. It is a well-known fact that $\mathcal{PS}(F_n) \subset \mathcal{BQ}(F_n)$, see \cite[Proposition 2.9]{LupiThesis}. In the case of $F_2$, Lupi \cite{LupiThesis}, in his thesis proves the equivalence of $\mathcal{PS}(F_2)$ and $\mathcal{BQ}(F_2)$ for $\pslr$-representations. Lee-Xu \cite{Lee-Xu} and Series \cite{CarolineSeriesPSBQ} independently proved that $\mathcal{BQ}(F_2) \subset \mathcal{PS}(F_2)$ for $\pslc$-representations of $F_2$ confirming a conjecture of Minsky in \cite{Minsky}. 

In \cite{Maloni}, the authors find a domain of discontinuity for $\text{MCG}(S_{0,4})$-action on relative $\slc$-character varieties of $S_{0,4}$, the four-holed sphere. This domain of discontinuity is given by the set of representations satisfying $SBQ$-conditions (see Definition~\ref{def:SBQ}). $SBQ$-conditions are just the restriction of $BQ$-conditions for words representing simple closed curves \cite[2.4]{Maloni}. Bowditch's conjecture \cite[Question C]{Bowditch1}, in this case, states that a type-preserving $\pslr$-representation is Fuchsian iff it satisfies $SBQ$-conditions. In this paper, we prove:

\begin{restatable}{theorem}{SBQconditions}\label{thm:discrete_sls}
	The holonomy of an admissible hyperbolic cone surface satisfies $SBQ$-conditions.
\end{restatable}

This proves that the hypothesis of being type-preserving in Bowditch's conjecture is necessary as these holonomies could be indiscrete.

For punctured spheres, we strengthen the definition of simple stability and introduce the notion of strong simple stability (see Section~\ref{StrongSimpleStability}). This is done by extending the simple-stability condition to separating curves, whereas in simple-stable representations, it applies only to non-separating simple curves. It follows from the proof of Theorem~\ref{thm:main1} that holonomy of an admissible hyperbolic cone surfaces is also strongly simple stable (see Proposition~\ref{prop:StrongSimpleStabilityAdmissibleSurface}). We prove the following analogous theorem for hyperbolic cone spheres:

\begin{restatable}{theorem}{PuncturedSpheresStrongSimpleStability}\label{thm:PuncturedSpheresStrongSimpleStability}
    The holonomies of admissible hyperbolic cone spheres are strongly simple-stable.
\end{restatable}

For surfaces with strictly more than one puncture, Theorem~\ref{thm:main1} exhibits a family of simple-stable representations that are not primitive-stable. For once-punctured surfaces, we have no examples of representations that are simple-stable but not primitive-stable. In particular, candidate examples such as the holonomies of hyperbolic cone surfaces are shown to be primitive-stable in the case of once punctured surfaces (see Theorem~\ref{thm:main1}). Constructing examples of simple-stable but not primitive-stable representations in this setting would be interesting. Furthermore, finding simple-stable representations beyond holonomies of hyperbolic cone surfaces is also an interesting problem. We pose the question as follows:

\begin{question}\label{Question1}
    Are there examples of simple-stable representations that are neither primitive-stable nor holonomies of hyperbolic cone surfaces?
\end{question}
The concept of simple-stability was originally introduced in the context of closed surfaces (see \cite{Minsky}), where it is readily verified that Fuchsian representations are simple-stable. Another interesting question is whether there exist non-Fuchsian representations that are simple-stable. This question is closely connected to the Bowditch conjecture mentioned above. 

In the case of the one-holed torus, we know that $\mathcal{PS}(F_2) = \mathcal{BQ}(F_2)$ (see \cite[Theorem I]{Lee-Xu}, \cite{CarolineSeriesPSBQ}). Let $\mathcal{SBQ}(S_{1,1})$ be the space of representations satisfying the $SBQ$-conditions. Then, for $\pslr$ representations, $\mathcal{PS}(F_2) = \mathcal{SS}(S_{1,1})$ and $\mathcal{BQ}(F_2) = \mathcal{SBQ}(S_{1,1})$, since primitive elements of $F_2$ are exactly the homotopy classes of non-peripheral simple closed curves on torus. Hence, we can conclude that $\mathcal{SS}(S_{1,1}) = \mathcal{SBQ}(S_{1,1})$ for the one-holed torus. Denote by $\mathcal{SSS}$, the set of strongly simple-stable representations. For every punctured surface, we know that $\mathcal{SSS} \subset \mathcal{SBQ}$. This leads naturally to the following question:

\begin{question}\label{Question2}
     Given a punctured surface $S_{g,n}$, is $\mathcal{SSS}(S_{g,n})$ same as $\mathcal{SBQ}(S_{g,n})$?
\end{question}

It is clear that strongly simple-stable representations are simple-stable but it remains uncertain whether simple-stable representations are also strongly simple-stable. Consequently, Question~\ref{Question2} raises the issue of whether simple-stable representations satisfy the $SBQ$-conditions.

\subsection*{Plan of the paper}In Section~\ref{sec:HyperbolicConeSurfaces}, we cover the basic facts about hyperbolic cone surfaces. Some of the results proved in this section can be found in \cite{PARLIER, SPTAN}. We have added details to certain proofs and included them for completeness. In Section~\ref{sec:CharacterVariety}, we provide the basics of relative $\pslr$-character varieties. In Section~\ref{sec:SimpleStability}, we define simple stability of representations and prove that $\mathcal{SS}(S_{g,n})$ forms a domain of discontinuity of $\text{MCG}(S_{g,n})$-action. In Section~\ref{sec:Holonomies}, we prove Theorems~\ref{thm:main1} and \ref{thm:PuncturedSpheresStrongSimpleStability}. In Section~\ref{sec:SBQconditions}, we define $SBQ$-conditions and see that Theorem~\ref{thm:discrete_sls} is a consequence of the proof of Theorem~\ref{thm:main1}. We provide an alternate proof of Theorem~\ref{thm:discrete_sls} independent of Theorem~\ref{thm:main1}.

\subsection*{Acknowledgements} I would like to express my heartfelt gratitude to my PhD advisor, Subhojoy Gupta, for introducing me to the works of Minsky and Goldman. I am deeply thankful for his invaluable insights, numerous discussions, and his thoughtful suggestions, which significantly improved the draft of this paper.
Special thanks are due to Siddhartha Gadgil for many insightful and fruitful discussions throughout the course of this work. I would also like to extend my appreciation to Gianluca Faraco for several engaging and productive conversations. Finally, I acknowledge the National Board for Higher Mathematics for their support through a fellowship, which made this work possible.

\section{Hyperbolic Cone Surfaces}\label{sec:HyperbolicConeSurfaces}

\begin{definition}
    A \textit{hyperbolic cone surface} is a surface obtained by pasting finitely many hyperbolic geodesic triangles. 	A vertex in a triangle of the triangulation is contained in more than one triangle. The \textit{cone-angle} of a vertex is the sum of all these vertex angles. A vertex whose cone-angle is not equal to $2 \pi$ is called a \textit{cone-point}. 
\end{definition}
\begin{example} \leavevmode
	\begin{enumerate}
		\item Any hyperbolic surface is a hyperbolic cone surface without any cone-points.
		\item Let $\mathcal{T}$ be a hyperbolic geodesic triangle with sides labeled $a,b, c$ and angles $\alpha,\beta,\gamma$. Consider another copy of $\mathcal{T}$, say $\mathcal{T}^\prime$, with sides labeled as $a', b', c'$. By gluing $\mathcal{T}$ and $\mathcal{T}^\prime$ along the corresponding sides (see Figure~\ref{fig:hypcone}), we obtain a hyperbolic cone sphere with 3 cone-points with cone-angles $2 \alpha, 2 \beta$ and $2 \gamma.$ This process of gluing is called doubling of a geodesic triangle. In fact, doubling of any hyperbolic geodesic polygon is a hyperbolic cone surface.
	\end{enumerate}
\end{example}

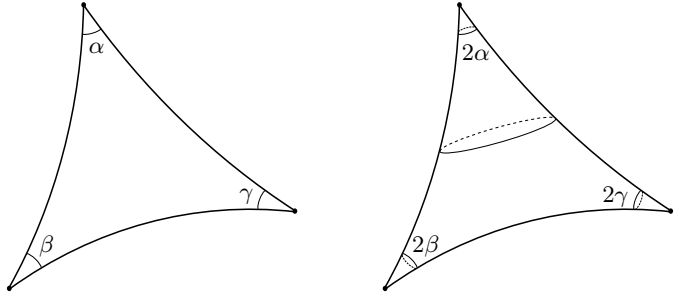
\begin{figure}[h]
            \labellist
            \small
            
            \pinlabel $\alpha$ at 90 245
		\pinlabel $\beta$ at 40 45
		\pinlabel $\gamma$ at 240 95

            \pinlabel $2\alpha$ at 470 240
		\pinlabel $2\beta$ at 420 45
		\pinlabel $2\gamma$ at 610 95
            
		\endlabellist
		\includegraphics[width=0.7\linewidth]{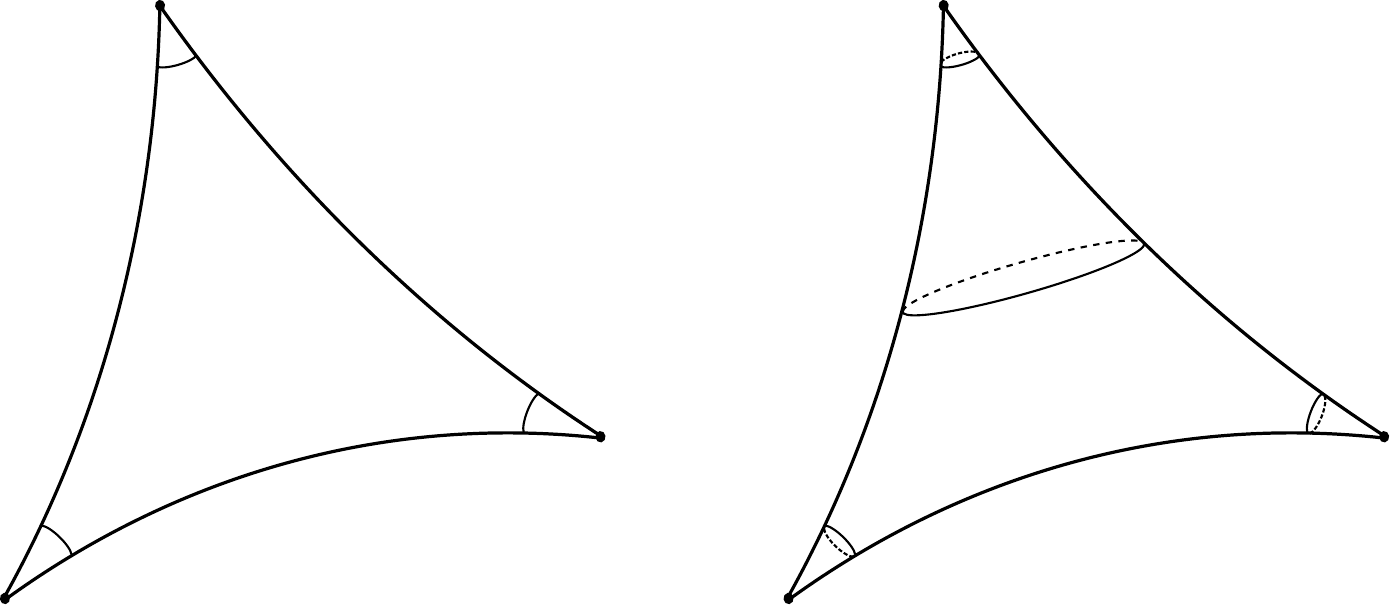}
            \caption{A cone sphere with 3 cone-points obtained by doubling the geodesic triangle on the left.}
		\label{fig:hypcone}
	\end{figure}

\noindent Note that if we remove cone-points from $S$, we get an incomplete hyperbolic surface whose metric completion is $S$.

\begin{definition} \label{HyperbolicCone}
	Let $\left( \mathbb{D}, d_{\mathbb{D}} \right)$ be the Poincar\'{e} disk equipped with the Poincar\'{e} distance function $d_{\mathbb{D}}$. For $\theta \in (0, 2 \pi )$ and $h \in \mathbb{R}^{+}$, define 
	$$\mathcal{S}_{h,\theta} = \{z \in \mathbb{D} ~ | ~ d_{\mathbb{D}}(0,z) \leq h \text{ and } 0 \leq \text{ arg }z \leq \theta\}.$$
	We define a \textit{hyperbolic cone} $\mathscr{C}_{h,\theta}$ with cone-angle $\theta$ and slant height $h$ to be the quotient of $\mathcal{S}_{h,\theta}$ obtained after gluing the geodesics $\text{arg }z=0$ and $\text{arg }z=\theta$ via the elliptic isometry $z \mapsto e^{i\theta}z$ fixing the origin.
\end{definition}

From now on, denote $S_{h,\theta} \setminus \{0\}$ by $\mathcal{S}^\ast_{h,\theta}$ and $\mathcal{C}_{h,\theta} \setminus \{0\}$ by $\mathcal{C}^\ast_{h,\theta}$. Consider the space $\left( \mathcal{S}^\ast_{h, \theta} \times \mathbb{Z} \right) / \sim$, where the equivalence relation is given as follows:
\[(r e^{i \theta},n) \sim (r, n+1)\]
\begin{proposition}
	$\widetilde{\mathcal{C}^\ast_{h, \theta}} \coloneqq \left( \mathcal{S}^\ast_{h, \theta} \times \mathbb{Z} \right) / \sim$ is a universal cover of $\mathcal{C}^\ast_{h, \theta}$.
\end{proposition}

\begin{proof}
	The map $\pi: \widetilde{\mathcal{C}^\ast_{h, \theta}} \rightarrow \mathcal{C}^\ast_{h,\theta}$ is the obvious map induced by the quotient map $q: \mathcal{S}_{h,\theta} \rightarrow \mathcal{C}_{h,\theta}$ i.e.,
	\[\pi(\widetilde{x},n) = q(\widetilde{x}).\]
	
	Now, we find the evenly covered neighbourhoods around every point in $\mathcal{C}_{h,\theta}^\ast$ (See Figure~\ref{fig:evennbhds_universalcovering}). Let $x \in \mathcal{C}^\ast_{h,\theta}$, then $q^{-1}(x)$ lies in $\mathcal{S}_{h,\theta}^\ast$. If $q^{-1}(x) = \widetilde{x}$ lies in the interior of $\mathcal{S}_{h,\theta}^\ast$, we can take the neighbourhood $U$ of $\widetilde{x}$ that lies completely inside the interior. If $q^{-1}(x) = \widetilde{x} \in \{(h, \alpha): 0 < \alpha < \theta \} \subset \mathcal{S}_{h,\theta}$, then take the half-disc around the $\widetilde{x}$ to be the neighbourhood $U$. If $q^{-1}(x) \in \{(r,\alpha): \alpha=0 \text{ or } \theta \}$, then observe that $x$ has two preimages $\widetilde{x}^1, \widetilde{x}^2$  then take the union of half-disks surrounding both $\widetilde{x}^1, \widetilde{x}^2$ to be the neighbourhood $U$. In all the cases, $U$ forms the evenly covered neighbourhood of $x$. Hence, $\pi$ is a covering map.
	
	\begin{figure}[h]
		\includegraphics[width=.4\linewidth]{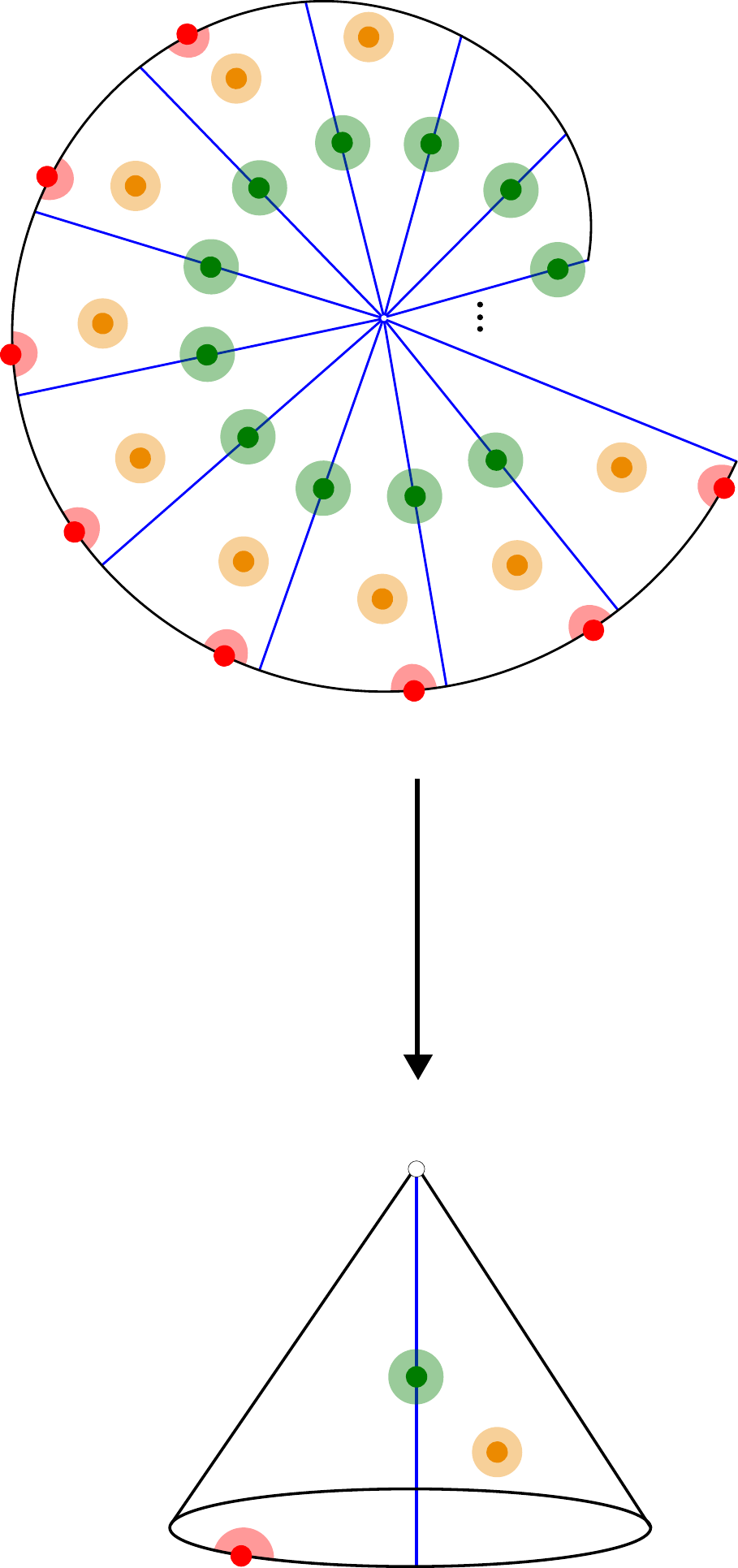}
		\captionof{figure}{$\pi: \widetilde{\mathcal{C}^\ast_{h,\theta}} \rightarrow \mathcal{C}^\ast_{h,\theta}$ }
		\label{fig:evennbhds_universalcovering}
	\end{figure}
 
	Now, since $\widetilde{\mathcal{C}^\ast_{h, \theta}}$ is an infinite degree covering of $\mathcal{C}^\ast_{h,\theta}$ and the fundamental group of $\mathcal{C}_{h,\theta}^\ast$ is $\mathbb{Z}$, $\widetilde{\mathcal{C}^\ast_{h, \theta}}$ is simply-connected.
\end{proof}

 Note that $\mathcal{C}^\ast_{h, \theta}$ is homeomorphic to an annulus and thus the universal cover $\widetilde{\mathcal{C}^\ast_{h, \theta}}$ is homeomorphic to $\mathbb{R}^2$. The metric completion $\overline{\widetilde{\mathcal{C}^\ast_{h,\theta}}}$ of $\widetilde{\mathcal{C}^\ast_{h,\theta}}$ will contain a point $\overline{0}$ with infinite cone-angle and will not be locally compact at that point \cite[Section 3.5]{BM}. We will think of $\widetilde{\mathcal{C}^\ast_{h,\theta}}$ as lying inside $\overline{\widetilde{\mathcal{C}^\ast_{h,\theta}}}$ as it would be helpful in visualising the cone-point. 
 
 Consider the homotopy class fixing endpoints of an arc $\sigma$ whose endpoints $\widetilde{x}, \widetilde{y}$ lie on the base of the cone $\mathcal{C}_{h,\theta}$ and $\sigma \cap \partial \mathcal{C}_{h,\theta} = \{\widetilde{x}, \widetilde{y}\}$. Any such homotopy class (upto a sign) of an arc is uniquely determined by the minimum number of times a representative of the homotopy class intersects the slant line of the cone transversely and without triple points. Call this the \textit{intersection number} of the homotopy class. We say that the arc is \textit{non-trivial}, if the arc cannot be homotoped to a single point on the boundary fixing endpoints. Fix $\widetilde{x}, \widetilde{y}$ on the base of the cone (not necessarily distinct). Denote by $\mathcal{H}_k$ the homotopy class of arcs with intersection number $k$.
 
 \begin{proposition}
     For every $k>0$, there exists a geodesic(length-minimising) arc $g_k$ realising the infimal length of the homotopy class $\mathcal{H}_k$.
 \end{proposition}

 \begin{proof}
 There are two cases to consider, one when the endpoints of the arc are the same and the other, when they are distinct.
 
 Suppose the endpoints are same. If the intersection number is zero, then the arc is trivial since the interior of the arc completely lies in the interior of $S_{h,\theta}^\ast \subset C_{h,\theta}^\ast$. For every other intersection number $k > 0$, length of every arc in the corresponding homotopy class $\mathcal{H}_k$ is bounded below by a non-zero value. This is because the shortest non-trivial arc starting and ending at the same point on the base of $\mathcal{C}_{h, \theta}$, is the geodesic joining the endpoints of the circular base of the sector $\mathcal{S}_{h,\theta}$ and it is dependent only on $\theta$ and $h$. Call the length of this geodesic $M_{\theta,h}$. For every $k$, there exists a sequence $\alpha^k_n$ in $\mathcal{H}_k$ of arcs such that $l(\alpha_k^n) \rightarrow \inf_{\sigma \in \mathcal{H}_k} l(\sigma) \geq M_{\theta,h}$. Now, by Arzela-Ascoli theorem, there exists an arc $g_k$ in $\mathcal{C}_{h,\theta}$ such that the infimum is realised.
 
 Suppose the endpoints are different. Firstly, every arc with different endpoints is non-trivial. The length of every arc in the homotopy class $\mathcal{H}_k$ is also bounded below by $M_{\theta^\prime}$, where $\theta^\prime$ is the angle subtended by the geodesic arc joining the endpoints of the circular arc at the center of the sector $\mathcal{S}_{h, \theta}$. The same argument follows and for every $k \geq 0$, there exists an arc $g_k$ such that the infimal length which is at least $M_{\theta^\prime}$ of a homotopy class is realised. 
 \end{proof}

  It is important to note here that in both the cases it is unclear whether $g_k$ belongs to the homotopy class $\mathcal{H}_k$. The following proposition gives a criteria for when $g_k$ belongs to $\mathcal{H}_k$.
 
 \begin{proposition} \label{thm:conearcs}
	Assume $\theta < \pi$, then $g_k$ does not pass through the cone-point if and only if $k < \left \lfloor \dfrac{\pi}{\theta} \right \rfloor$.
 \end{proposition} 

\begin{proof}
	Suppose $x,y$ be the endpoints of an arc $\gamma_k \in \mathcal{H}_k$, not necessarily distinct. Let $\widetilde{\gamma_k}$ be the lift of $\gamma_k$ into $\widetilde{\mathcal{C}^\ast_{h, \theta}}$, such that $\widetilde{\gamma_k}(0) = \widetilde{x} \in \mathcal{S}^\ast_{h, \theta} \times 0$. Then, as $\gamma_k \in \mathcal{H}_k$, $\widetilde{\gamma_k}(1) = \widetilde{y} \in \mathcal{S}^\ast_{h, \theta} \times k$. Now, the shortest path between $\widetilde{x}$ and $\widetilde{y}$ would give us the infimal path below.

    \begin{figure*}[h!]
    \centering
        \labellist
		\Large
		\pinlabel $\widetilde{x}$ at 475 100
		\pinlabel $\theta$ at 300 25
		\pinlabel $\widetilde{y}$ at 40 195
            \pinlabel $\widetilde{x}$ at 1070 215
		\pinlabel $\theta$ at 940 175
		\pinlabel $\widetilde{y}$ at 700 10
		\endlabellist
        \begin{subfigure}[t]{0.5\textwidth}
            \centering
            \includegraphics[height=1in]{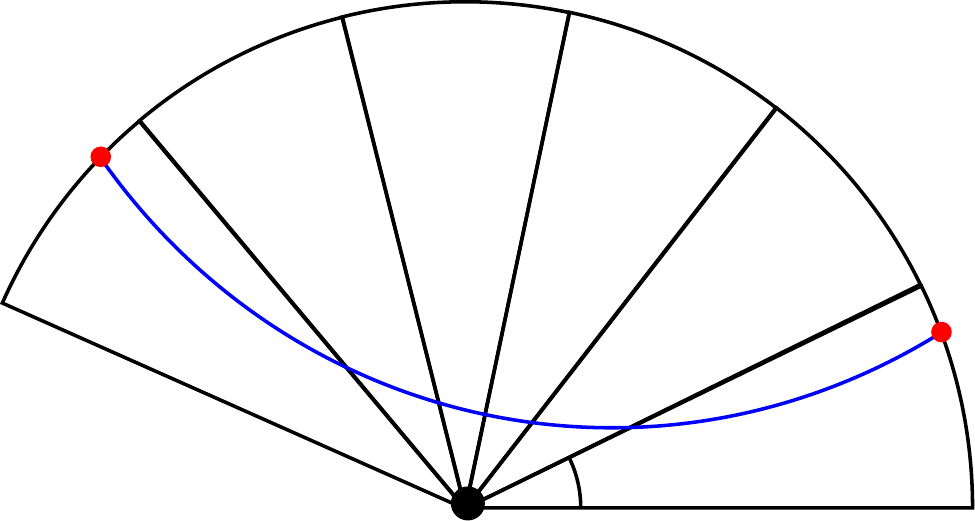}
            \caption{$k < \left \lfloor \dfrac{\pi}{\theta} \right \rfloor $}
            \label{fig:convex}
        \end{subfigure}%
    ~ 
        \begin{subfigure}[t]{0.5\textwidth}
            \centering
            \includegraphics[height=1.4in]{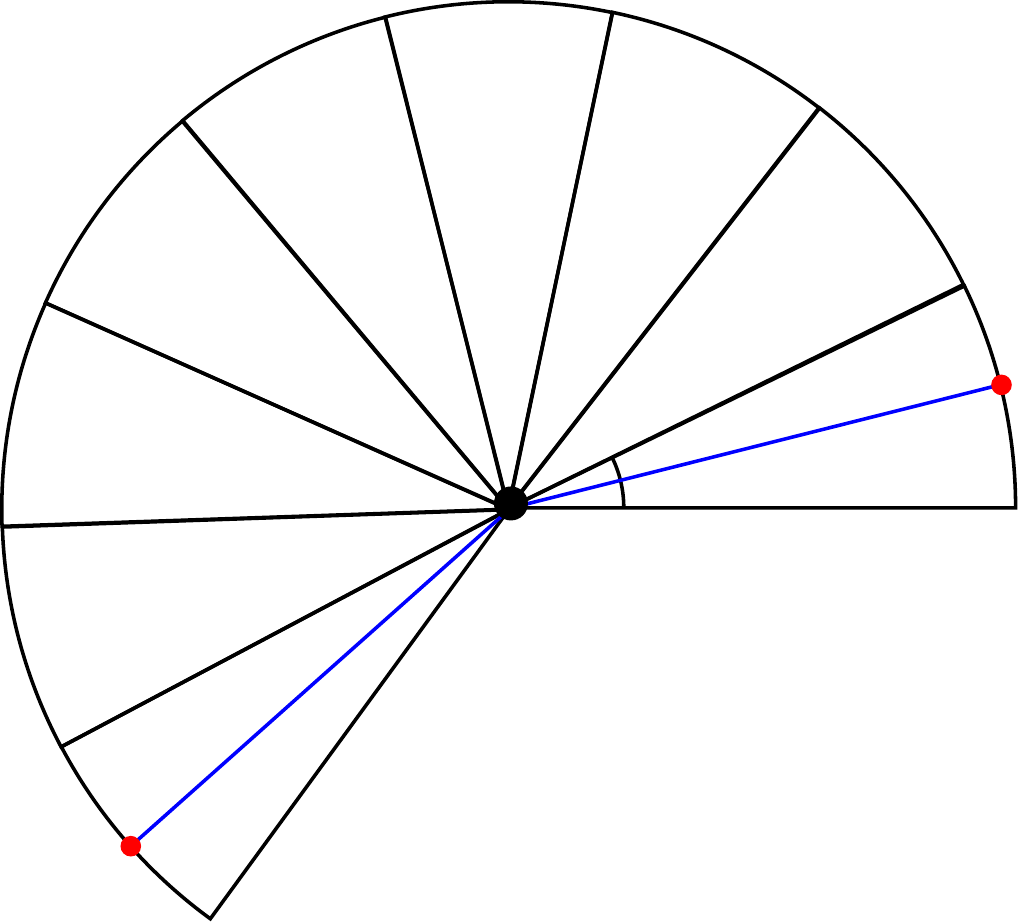}
            \caption{$k \geq \left \lfloor \dfrac{\pi}{\theta} \right \rfloor $}
            \label{fig:non-convex}
        \end{subfigure}
        \caption{Shortest paths joining $\widetilde{x}, \widetilde{y}$}
    \end{figure*}
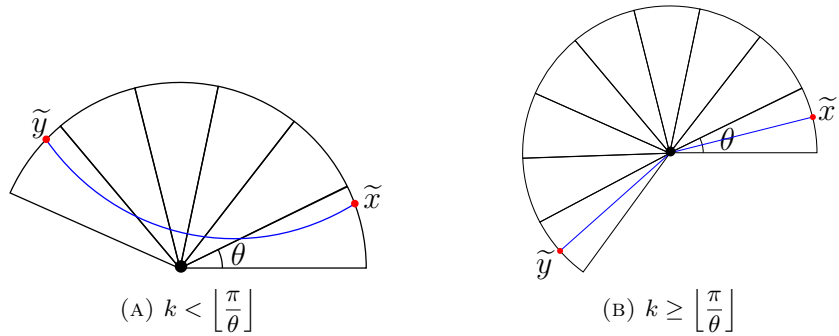
    

	
	Suppose $k < \left \lfloor \dfrac{\pi}{\theta} \right \rfloor$, then $\widetilde{x}$ and $\widetilde{y}$ lie in a convex subset of $\widetilde{\mathcal{C}^\ast_{h, \theta}}$, thus there is a geodesic joining them away from the cone-point and thus the infimal arc $g_k$ stays away from the cone-point below. See Figure~\ref{fig:convex}.
	
	Now, suppose $k \geq \left \lfloor \dfrac{\pi}{\theta} \right \rfloor$, then the shortest path joining $\widetilde{x}$ and $\widetilde{y}$ is the one passing through the cone-point and thus the infimal arc $g_k$ has to pass through the cone-point below. See Figure~\ref{fig:non-convex}.
\end{proof}

\begin{remark}\label{rem:HomotopyClassPreserved}
	From the proof above, we make the following observations:
	\begin{enumerate}
		\item In the case when $k < \left \lfloor \dfrac{\pi}{\theta} \right \rfloor$, the infimal geodesic actually lies in the homotopy class $\mathcal{H}_k$.
		\item For every $k \geq \left \lfloor \dfrac{\pi}{\theta} \right \rfloor$, the infimal geodesic is exactly the same and the infimal length is $2h$, where $h$ is the slant height of the cone.
	\end{enumerate}
\end{remark}

\begin{lemma}\label{lem:cone_nbhd}
    Let $\mathcal{C}_{h, \theta}$ be a hyperbolic cone where $\theta < \pi$. There exists $d > 0$ and a cone neighbourhood $U_d$ of the cone-point with the following property:
    \begin{center}
        If $x, y$ are two points at distance $d$ from the cone-point, then the geodesic joining $x,y$ does not intersect $U_d$.   
    \end{center}
       
\end{lemma}    
 \begin{proof}
    
    Let $x,y \in \mathcal{C}_{h,\theta}$ such that $d\coloneqq d_{\mathbb{H}^2}(0, x) = d_{\mathbb{H}^2}(0, y)$ and $\gamma$ be a geodesic joining $x,y$. Fix a point $z$ in the base of the cone such that the slant line $l$, i.e., the geodesic joining the cone-point $0$ to $z$ does not intersect $\gamma$. Now, observe that $\mathcal{C}_{h,\theta} \setminus l$ is a sector $\mathcal{S}_{h,\theta}$ (see Figure~\ref{fig:neighbourhood1}). 
    
    Let $\sigma_d$ be the geodesic joining the points $(d,0)$ and $(d, \theta)$. We see that the lift of the geodesic arc $\gamma$ joining $x, y$, i.e., $\widetilde{\gamma}$ joining $\widetilde{x}$ and $\widetilde{y}$, does not intersect $\sigma_d$. Define $r \coloneqq d_{\mathbb{H}^2}(0,\sigma_d)$ and $U_d = \{\widetilde{x} \in \mathcal{S}_{h,\theta}~\vert~ d_{\mathbb{H}^2}(0,\widetilde{x}) < r\}$. Note that the neighbourhood $U_d$ is independent of the choices of $x,y,z$ and only depends on $d$ (see Figure~\ref{fig:neighbourhood2}). 
    \begin{figure}[h]
		\labellist
		\pinlabel $x$ at 120 90 
            \pinlabel $y$ at 210 90
            \pinlabel $z$ at 155 65
		\pinlabel $\gamma$ at 155 160
            \pinlabel $\left(d,\theta \right)$ at 410 145
            \pinlabel $(d,0)$ at 670 8
            \pinlabel $\sigma_d$ at 540 90
            \pinlabel $U_d$ at 490 40
            \pinlabel $\widetilde{y}$ at 545 208
            \pinlabel $\widetilde{x}$ at 660 140
		\endlabellist
		\centering
		\centering
		\includegraphics[width=0.7\linewidth]{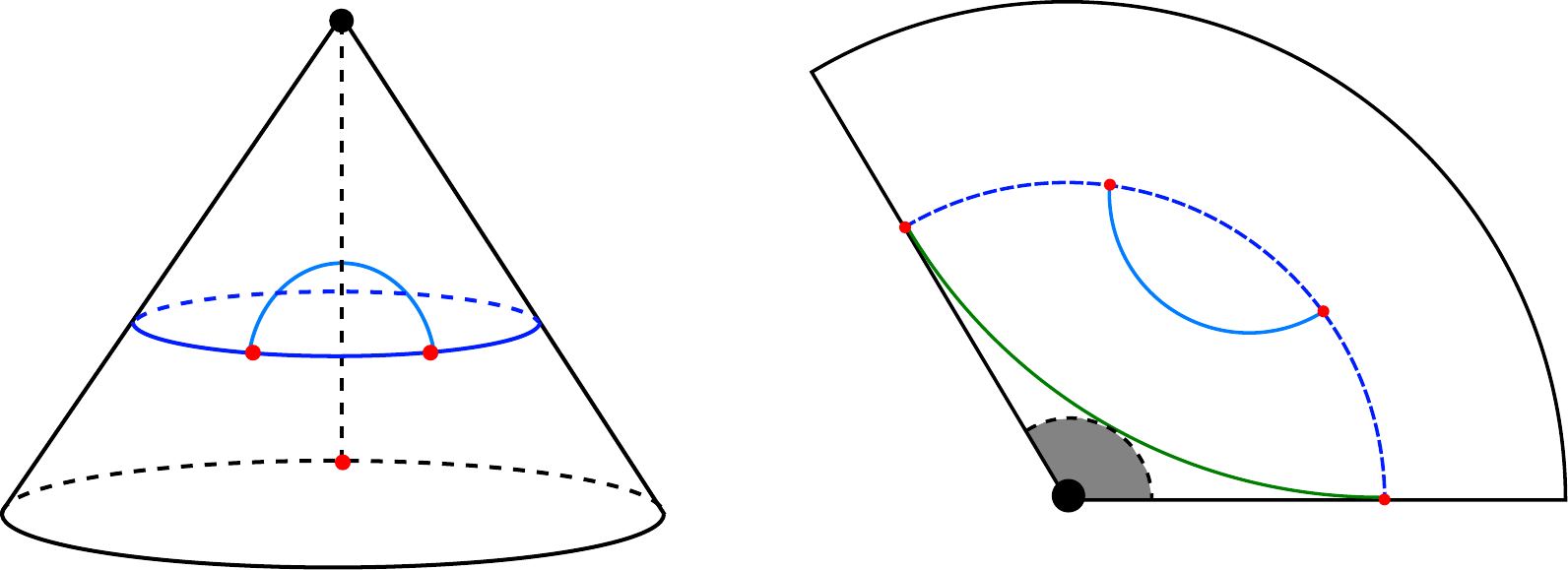}
		\captionof{figure}{The geodesic $\gamma$ in $\mathcal{C}_{h,\theta}$ and $\mathcal{S}_{h,\theta}$. }
		\label{fig:neighbourhood1}
	\end{figure}
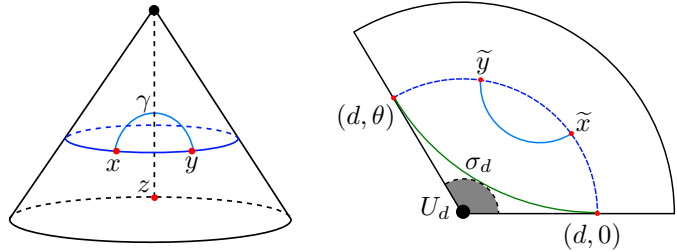

     \begin{figure}[h]
		\labellist
		\pinlabel $x$ at 110 80
            \pinlabel $y$ at 185 80
            \pinlabel $z$ at 137 58
            \pinlabel $U_d$ at 190 210
		\endlabellist
		\centering
		\centering
		\includegraphics[width=0.25\linewidth]{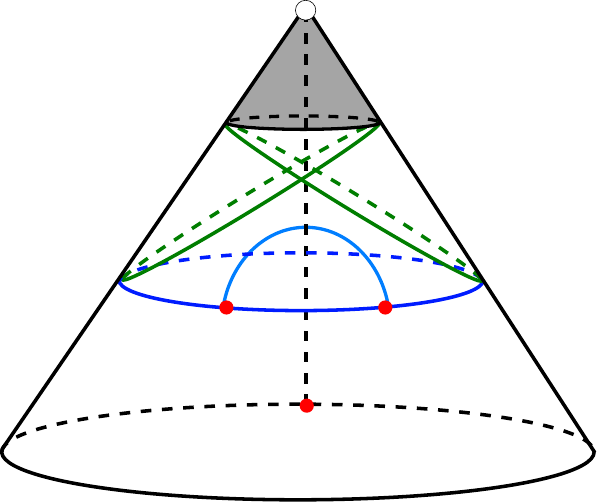}
		\captionof{figure}{The cone neighbourhood $U_d$.}
		\label{fig:neighbourhood2}
	\end{figure}
 \end{proof}   

\begin{definition}
    A hyperbolic cone surface is said to be \textit{admissible} if all the cone-points have cone-angles strictly less than $\pi$.
    
\end{definition}

Now, we will move on to proving some basic facts about hyperbolic cone surfaces. One can find some of these results in \cite{PARLIER} and \cite{SPTAN}. We elaborate on certain proofs and have slightly modified statements for theorems. 
\begin{proposition}\label{prop:selfintersection}
    Any closed curve on a punctured surface that goes around the puncture more than once has a self-intersection. 
\end{proposition}
\begin{proof}
    Equip the punctured surface with a complete finite-area hyperbolic surface, i.e., a hyperbolic surface with finitely many cusps corresponding to punctures. Let $\gamma$ be the closed curve that goes around the puncture more than once. This implies that in the universal cover a lift of $\gamma$ intersects two adjacent sides in the fundamental polygon with vertex at infinity.  Therefore, by \cite[Lemma 5.1]{BirmanSeries}, $\gamma$ has a self-intersection. 
\end{proof}
\begin{definition}
    A curve on $S$ is called \textit{cone-point peripheral} if it bounds a disk with cone-points and if it bounds only one cone-point we call it just \textit{peripheral}. 
\end{definition}

\begin{proposition} \label{prop:ShortCurves}
    There exists $M > 0$ such that any closed curve in $S \setminus P$, where $P$ is the set of cone-points, with length less than $M$ is either trivial or peripheral. 
\end{proposition}

\begin{proof}
    As there are only finitely many cone-points, we can choose a $\delta > 0$ such that cone neighbourhoods of radius $\delta$ of distinct cone-points do not intersect. (In \cite{PARLIER}, the authors prove the stronger statement that if cone-angles are strictly less than $\pi$, then we can find cone neighbourhoods with slant heights depending only on cone-angles. In other words, distinct cone-points cannot be arbitrarily close if cone-angles are strictly less than $\pi$.)
    
    Now, consider the compact subset $K$ of $S$ obtained after removing disjoint cone neighbourhoods around cone-points. The injectivity radius is a continuous function on $S$, therefore, it will have an infimum on $K$, say $\frac{M}{2}$. Let $\gamma$ be closed curve in $S \setminus P$ of length less than $M$. Suppose $\gamma \cap K = \emptyset$, then it lies completely inside a cone neighbourhood, by connectedness of $\gamma$. This implies that it is either trivial or homotopic to a cone-point. Suppose $\gamma \cap K \neq \emptyset$, then there exists an embedded disk $B_{M/2}(x)$ around an $x \in \gamma \cap K$ such that $\gamma \subset B_{M/2}(x)$. This implies that $\gamma$ is trivial.
\end{proof}
\begin{theorem} \label{thm:scgeodesic_existence}
Let $S$ be an admissible hyperbolic cone surface. Let $\gamma$ be an essential simple closed curve in $S \setminus P$ i.e., non-peripheral and non-trivial simple closed curve, where $P$ is the set of cone-points. Then, there exists a unique simple closed geodesic representative $c$ in the free homotopy class of $\gamma$ in $S \setminus P$. 
\end{theorem} 
\begin{proof}
	Let $\mathscr{H}$ denote the free homotopy class of $\gamma$ in $S \setminus P$. As $\gamma$ is an essential simple closed curve, $l(\sigma) > M > 0$ for all $\sigma \in \mathscr{H}$, where $M$ is the value from Proposition~\ref{prop:ShortCurves}. Hence, there exists $l > M$ such that $\inf_{\sigma \in \mathscr{H}} l(\sigma) = l $. 
	
	Now, let $\{\sigma_n\}_n$ be a sequence of curves such that $\{l(\sigma_n)\}$ is a decreasing sequence converging to $l$. Reparametrise these curves to be defined on the same interval with speed proportional to arc-length. We have an equicontinuous family of functions into the compact metric space $S$, thus, by Arzela-Ascoli, there exists a $c$ such that $\sigma_n \rightarrow c$ uniformly.
    
    Suppose $c$ passes through a cone-point $p$. Let $\mathcal{C}_p$ be a cone neighbourhood around the cone-point $p$ and $\alpha_1, \ldots, \alpha_l$ be the geodesic arcs $c \cap \mathcal{C}_p$. If $\alpha_k$ passes through the cone-point $p$, we can decrease the length of $\alpha_k$ by homotoping it away from $p$. This is because the cone-angles are less than $\pi$. As $\mathcal{H}$ is a homotopy class of a simple closed curve, it follows from Proposition~\ref{prop:selfintersection} that any of the arcs $\alpha_i$ cannot go around the puncture more than once. Thus, homotoping away from $p$, will preserve the homotopy class $\mathcal{H}$ (see Remark~\ref{rem:HomotopyClassPreserved}). Hence, $c \in \mathcal{H}$ and $c$ cannot pass through the cone-points. Also, as $c$ is locally length minimising, it is a geodesic.

	
	Now, we will prove that $c$ is simple. For brevity, denote $S \setminus P$ by $\mathcal{S}$. Let $\widetilde{\mathcal{S}}$ be the universal cover of $\mathcal{S}$. Note that $\widetilde{\mathcal{S}}$ is homeomorphic to the Riemannian universal cover of $X$, i.e., $S \setminus P$ with a fixed complete hyperbolic metric. Now, as $S \setminus P$ has a hyperbolic structure, we get a developing map $\text{dev}:\widetilde{\mathcal{S}} \rightarrow \mathbb{H}^2$. The map $\text{dev}$ is a local isometry. 
	
	\begin{figure}[h]\labellist
            \pinlabel $\widetilde{S}$ at 20 650
            \pinlabel $X$ at 800 600
            \pinlabel $\mathcal{S}$ at 800 300
		\pinlabel $\mathbb{H}^2$ at 20 160
		\endlabellist
		\centering
		\centering
		\includegraphics[width=8cm]{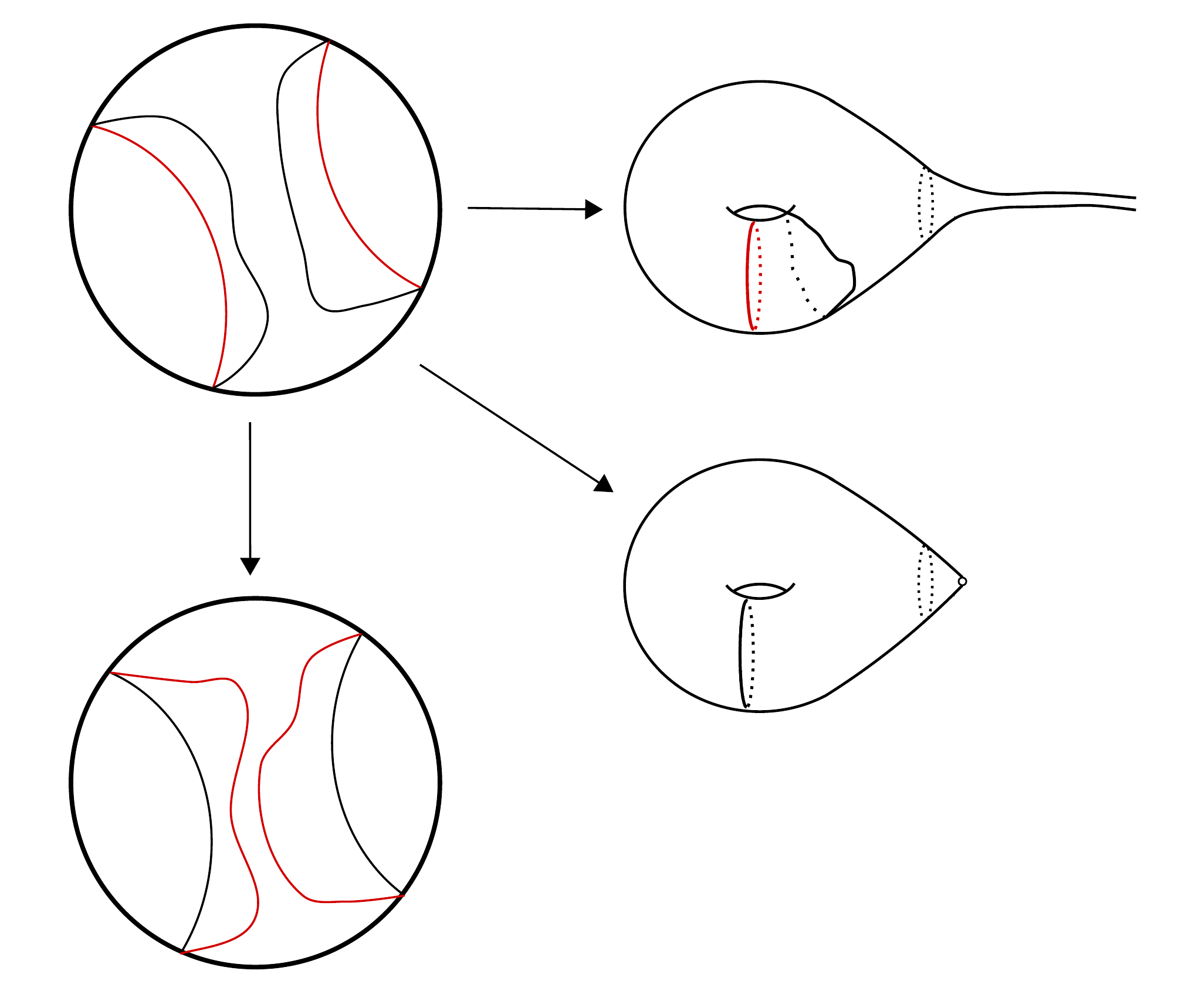}
		\caption{Lifts into $\widetilde{S}$ and $\mathbb{H}^2$, where $\widetilde{S}$ is viewed as the same topological space $\mathbb{H}^2$}
		\label{fig:simplicity}
	\end{figure}
	  
	Let $\widetilde{c}$ be a lift of $c$ into $\widetilde{\mathcal{S}}$. Suppose $\widetilde{c}$ has a transverse self-intersection, then there exists $\widetilde{x} \in \widetilde{\mathcal{S}}$ such that two arcs of $\widetilde{c}$ passes through $\widetilde{x}$ with distinct tangents. Observe that $\text{dev}(\widetilde{c})$ is a hyperbolic geodesic, since $\widetilde{c}$ is a geodesic and $\text{dev}$ is a local isometry. This gives a contradiction, since at $\text{dev}(\widetilde{x})$, the geodesic $\text{dev}(\widetilde{c})$ cannot have distinct tangents. Hence, $\widetilde{c}$ is simple.
	
	 Let $g$ be the unique simple closed geodesic in the homotopy class of $\mathscr{H}$ in complete hyperbolic surface $X$. As $g$ is simple, two distinct lifts of $g$ do not intersect in $\widetilde{\mathcal{S}}$ and they have distinct endpoints that are nested. Also, we know that the lifts $\widetilde{g}$ and $\widetilde{c}$ obtained by lifting the homotopy between $g$ and $c$ have the same endpoints at infinity. This implies that, distinct lifts of $c$ have distinct endpoints that are nested.
	
	As the endpoints are nested, two distinct lifts of $c$, if they intersect transversally must form a bigon, which implies that their developing images must form a bigon, which is impossible since they are hyperbolic geodesics. Hence, no two lifts of $c$ intersect.
    
    Also, by \cite[Proposition 1.4]{FB}, $c$ is not a power of a closed curve, since it represents the conjugacy class of a simple closed curve. In summary, we see that $c$ is not a power of a closed curve, every lift of $c$ is simple and no two distinct lifts intersect. Hence, $c$ is simple. 
	
    Now, we prove the uniqueness. Suppose $c, d$ are two simple closed geodesic representatives of $\gamma$. Now, if $c, d$ intersect, then they form a bigon, which implies that the developing images of their lifts form a bigon. This is impossible because hyperbolic geodesics cannot form a bigon. Therefore, $c \cap d = \emptyset$. This implies that there is an embedded hyperbolic cylinder with $c$ and $d$ as its boundary circles. This is again a contradiction, since by Gauss-Bonnet there is no hyperbolic cylinder with two geodesic boundaries.
\end{proof}
\begin{remark}{\ }
    \begin{enumerate}
        \item The condition that the cone-angles are less than $\pi$ is necessary because homotoping across the cone-point might not preserve the homotopy class of the curve (see Remark~\ref{rem:HomotopyClassPreserved}).
        \item In the case of two cone-points of cone-angle $\pi$, the geodesic representative of the simple closed curve bounding these two cone-points degenerate to the geodesic arc between these cone-points.
    \end{enumerate}    
\end{remark}
\begin{remark}
    We have actually proved something slightly stronger than the statement above. Let $\theta_1, \ldots ,\theta_n$ be the cone-angles on the surface and $k_i = \left \lfloor \frac{\pi}{\theta_i} \right \rfloor$. Then, the geodesic corresponding to any curve such that arcs lying in the cone neighbourhood of $p_i$ do not wrap around the cone-point more than $k_i$ times cannot pass through $p_i$.
\end{remark}

Now, we will prove an important corollary of the above theorem which will be used in the next section. Recall that the fundamental group of $S_{g,n}$ is given as follows:
\[\pi_1(S_{g,n}) = \left \langle a_1, b_1, \ldots, a_g,b_g, c_1, \ldots, c_n \Big \vert \prod_{i=1}^g [a_i,b_i] \prod_{j=1}^n c_j \right \rangle\]

\begin{proposition}\label{prop:admissiblePolygon}
Any admissible hyperbolic cone surface $S$ can be obtained from a convex hyperbolic polygon $\mathcal{P}$ by identifying sides via isometries. 
\end{proposition}

\begin{proof}
    Suppose the admissible hyperbolic surface has genus $g$ and $n$ cone-points. We will cut our surface along a set $\{a_1, b_1, \ldots, a_g,b_g\}$ of simple closed lassos (geodesics except possibly at one point) based at a point $x \in S$ and a set of simple geodesic arcs $\{l_1, \ldots, l_n\}$ starting at $x$ to obtain a convex hyperbolic polygon. We will construct the required lassos and arcs below.
    
    Firstly, there exists two simple closed curves, say $a_1^\prime, b_1^\prime$ such that they intersect exactly once on $S$. Now, by Theorem~\ref{thm:scgeodesic_existence}, $a_1^\prime, b_1^\prime$ can be homotoped to simple closed geodesics $a_1, b_1$ respectively such that they intersect at one point, say $x$. Let $l_1, \ldots, l_n$ be the family of disjoint simple geodesic arcs joining the cone-points to the point $x$. Now, cutting the surface $S$ along the curves $a_1, b_1$ and the arcs $l_1, \ldots, l_n$, we obtain a hyperbolic surface $S^\prime$ with piecewise geodesic boundary and internal cone-angles less than $\pi$ (see Figure~\ref{fig:convex_polygon}).

    We will construct the remaining simple closed lassos inductively. We already have finished constructing $a_1, b_1$, which are in fact simple closed geodesics. Now, suppose $a_1, b_1, \ldots, a_{k-1}, b_{k-1}$ have been constructed on $S$. Let $a_k^\prime, b_k^\prime$ be two simple closed curves based at $x$ on $S$ such that:
    \begin{itemize}
        \item The homotopy classes of $a_k^\prime, b_k^\prime$ is distinct from $a_1, b_1, \ldots, a_{k-1}, b_{k-1}$.
        \item The homotopy classes of $a_k^\prime$ and $b_k^\prime$ has geometric intersection number 1.
        \item They have representatives disjoint from $a_1, b_1, \ldots, a_{k-1}, b_{k-1}$.
    \end{itemize}
    
    Note that even though they have disjoint representatives in their free homotopy classes, the representatives we are interested in must intersect at $x$. Hence, the image of these curves in $S^{\prime}$ is a family of arcs starting and ending at the cone-points on the boundary corresponding to $x$. Now, from \cite[Lemma 3.6]{DESPRE}, we know that the geodesic representatives of simple disjoint arcs with endpoints on the geodesic boundary, remain simple and disjoint. Hence, geodesic representatives $a_k, b_k$ of $a_k^\prime, b_k^\prime$ respectively are such that they intersect $a_1, b_1, \ldots, a_{k-1}, b_{k-1}$ only at $x$. Repeating this construction, we get a family of simple closed lassos $a_1, b_1, \ldots, a_g,b_g$ and geodesic arcs $l_1, \ldots, l_n$ such that cutting along these gives a convex hyperbolic polygon.

     \begin{figure}[h]
		\labellist
            \tiny
		\pinlabel $x$ at 455 620
            \pinlabel $a_1$ at 415 580
            \pinlabel $b_1$ at 400 480
		\endlabellist
		\centering
		\centering  
		\includegraphics[width=0.9\linewidth]{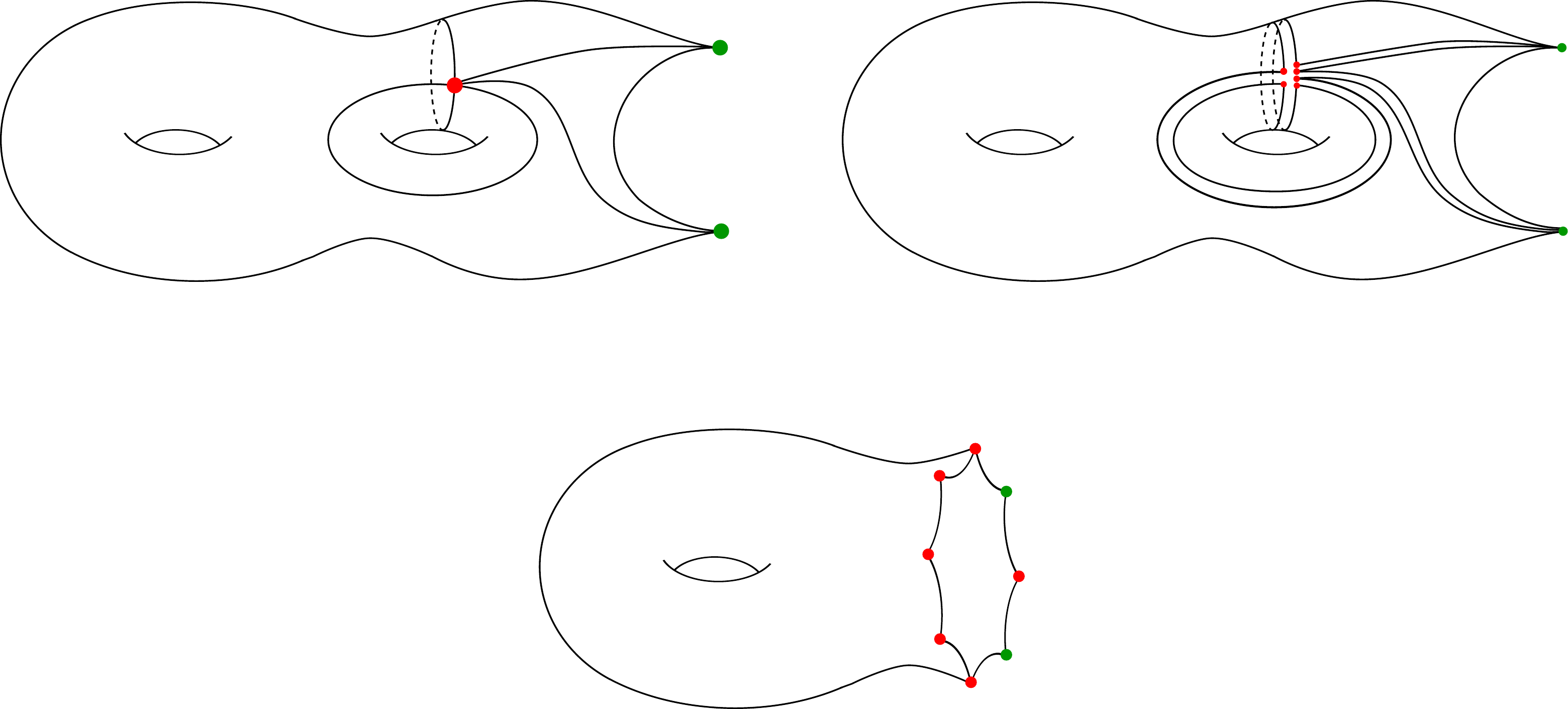}
		\captionof{figure}{Simple geodesic lassos and simple arcs}
		\label{fig:convex_polygon}
	\end{figure}
\end{proof}

\begin{remark}\label{rem:UniversalCover}
    We can construct a Riemannian universal cover $\widetilde{S \setminus P}$ for $S \setminus P$ using the polygon $\mathcal{P}$, obtained above, in the usual way. Let $\{s_1,\ldots,s_{4g+2n}\}$ be the sides of the polygon $\mathcal{P}$ such that there exists ordered pairs of sides $(s_i,s_j)$ such that one of the generators, say $g_{ij}$, of the fundamental group identifies one with the other. Therefore, the universal cover of $S \setminus P$ is given by $\widetilde{S \setminus P} \coloneqq \left(\mathcal{P} \setminus \{p_1,\ldots,p_n\}\right) \times \pi_1(S_{g,n})/ \sim$, such that $(s_i, \gamma) \sim (s_j,\gamma \cdot g_{ij})$ and $p_1, \ldots, p_n$ are the vertices corresponding to the cone-points. We denote this covering map by $\pi: \widetilde{S \setminus P} \rightarrow S \setminus P$. The metric on $\widetilde{S \setminus P}$ is the induced metric. Note that this metric is incomplete as we can have Cauchy sequences that converge to the removed vertices. 
\end{remark}

\section{$\pslr$-Character Variety and Mapping Class group action}\label{sec:CharacterVariety}
Let $S_{g,n}$ denote a surface of genus $g$ with $n$ punctures. The fundamental group of $S_{g,n}$ is given by the following presentation of the fundamental group:
 \[\pi_1(S_{g,n}) \coloneqq \left \langle a_1, b_1, \ldots, a_g, b_g, c_1, \ldots, c_n \Big \vert \prod_{i=1}^g \left[a_i, b_i\right] \prod_{k=1}^n c_k  \right \rangle\]
 
Denote by $\mathcal{A}$, the generating set $\{a_1, b_1, \ldots, a_g, b_g,c_1,\ldots,c_n\}$ of $\pi_1(S_{g,n})$. Also, observe that $\pi_1(S_{g,n})$ is isomorphic to $F_k$, where $F_k$ is a free group on $k$ letters and $k = 2g+n-1$. 

The \textit{$\pslr$-representation variety} of $\pi_1(S_{g,n})$ is defined as the set of all homomorphisms from $\pi_1(S_{g,n})$ to $\pslr$ and denoted by $\text{Hom}(\pi_1(S_{g,n},\pslr) $. Note that  $\text{Hom}(\pi_1(S_{g,n},\pslr)$ is isomorphic to $\left(\pslr\right)^{2g+n-1}$, since a homomorphism is uniquely determined by the images of the generators of $\pi_1(S_{g,n})$.

There exists a natural action of $\pslr$ on $\text{Hom}(\pi_1(S_{g,n}), \pslr)$ which is just the conjugation action i.e., 
\[A \cdot \rho = A \rho A^{-1}, \forall A \in \pslr \]
Also, note that there is an action of $\text{Aut}(\pi_1(S_{g,n}))$ on $\text{Hom}(\pi_1(S_{g,n}), \pslr)$ given by
\[\phi \cdot \rho = \rho \circ \phi^{-1}, \forall \phi \in \text{Aut}(\pi_1(S_{g,n}))\]

\begin{definition}
    A representation $\rho: \pi_1(S_{g,n}) \rightarrow \pslr$ is said to be \textit{non-elementary} if the subgroup $\rho(\pi_1(S_{g,n}))$ does not fix any finite set in $\overline{\mathbb{H}^2}$. 
\end{definition}

We now, restrict to the $\pslr$-invariant subspace, $\mathrm{Hom}^\text{ne}(\pi_1({S_{g,n}}), \pslr)$, of non-elementary representations. We define the \textit{$\pslr$-character variety} of $\pi_1({S_{g,n}})$ as the orbit space of $\text{Hom}^\text{ne}(\pi_1({S_{g,n}}), \pslr)$ under the conjugation action of $\pslr$ i.e.,
\[\mathcal{X}(S_{g,n}) \coloneqq \text{Hom}^{\text{ne}}(\pi_1(S_{g,n}), \pslr) \big / \pslr\]

Note that the action of $\text{Aut}(\pi_1(S_{g,n}))$ gives rise to an action of $\text{Out}(\pi_1(S_{g,n})) \coloneqq \text{Aut}(\pi_1(S_{g,n}))/ \text{Inn}(\pi_1(S_{g,n}))$ on $\mathcal{X}(S_{g,n})$. This is because, if $\phi \in \text{Inn}((\pi_1(S_{g,n}))$, then $\phi(\gamma) = \sigma \gamma \sigma^{-1}, \forall \gamma \in \pi_1(S_{g,n})$ and 
\[\phi \cdot \rho (\gamma) = \rho \circ \phi^{-1}(\gamma) = \rho(\sigma)^{-1} \rho(\gamma) \rho(\sigma)\]
This implies that $\rho \sim \phi \cdot \rho, \forall \phi \in \text{Inn}(\pi_1(S_{g,n}))$.

For every $\gamma \in \pi_1(S_{g,n})$, we have a character map $\tau_{\gamma}: \text{Hom}(\pi_1(S_{g,n}),\pslr) \rightarrow \mathbb{R}$ defined as follows:
\[\tau_\gamma (\rho) \coloneqq \tra^2(\rho(\gamma)).\]
The ring of all character maps $T$ is finitely generated \cite[Theorem~1.4.1]{CULLER-SHALEN}. Let $\tau_{\gamma_1}, \ldots, \tau_{\gamma_k}$ be thus obtained set of generators.  We can define a vector-valued function $t: \text{Hom}(\pi_1(S_{g,n}), \pslr) \rightarrow \mathbb{R}^k$ as $t(\rho) \coloneqq (\tau_{\gamma_1}(\rho), \ldots, \tau_{\gamma_k}(\rho))$. Note that this map factors through orbit space under conjugation as traces are conjugation invariant. The character variety can also be interpreted as a subspace of the image of non-elementary representations under $t$ (see \cite{DanielMathews1}). For various other definitions of character varieties, see \cite{Maret}.




Let $\mathcal{C} = \{C_1, \ldots, C_n\}$ be a set of conjugacy classes in $\pslr$. We define the \textit{relative $\pslr$-representation variety} associated to the set $\mathcal{C}$ as follows:
\[\text{Hom}^{\text{ne}}_{\mathcal{C}}(\pi_1(S_{g,n},\pslr)) \coloneqq \left\{ \rho \in \text{Hom}^{\text{ne}}(\pi_1(S_{g,n}, \pslr)): \rho(c_i) \in C_i, \forall i\right\}\]

Now, we define the \textit{relative $\pslr$-character variety} of $S_{g,n}$ associated to the set $\mathcal{C}$ as the quotient of the relative representation variety by the same conjugation action of $\pslr$ 
\[\mathcal{X}(S_{g,n}, \mathcal{C}) \coloneqq \text{Hom}^{\text{ne}}_{\mathcal{C}}(\pi_1(S_{g,n},\pslr)) \big /\pslr\]

Note that $\text{Aut}(\pi_1(S_{g,n}))$ does not preserve relative representation varieties. However, the subgroup $\text{Aut}_\mathcal{C}^{\ast}(\pi_1(S_{g,n}))$ i.e., the elements of $\text{Aut}(\pi_1(S_{g,n}))$ that fixes the set of conjugacy classes $\mathcal{C}$, also preserve $\text{Hom}^{\text{ne}}_{\mathcal{C}}(\pi_1(S_{g,n},\pslr))$. Hence, $\text{Out}_{\mathcal{C}}^{\ast}(\pi_1(S_{g,n}))$ preserves $\mathcal{X}(S_{g,n}, \mathcal{C})$. 

The holonomy of hyperbolic cone surfaces with prescribed cone-angles lie in one of these relative character varieties. In the case of the one-holed torus, from \cite[Theorem 3.4.1]{GOLDMAN}, we see that every representation in $\mathcal{\mathcal{X}}(S_{1,1}, \mathcal{C})$ where $\mathcal{C}$ is an elliptic element, is a holonomy of hyperbolic cone torus with one cone-point of cone-angle less than $2 \pi$. For punctured spheres, we know that supra-maximal representations (see \cite{DT}) send every homotopy class representing a simple closed curve to elliptic elements, thus cannot be a holonomy of admissible hyperbolic cone surfaces. In general, following the argument in \cite[Section 5]{TLEFILS}, it can be proved that the holonomy of hyperbolic cone surfaces with prescribed cone-angles form an open subset of relative character varieties. 

\begin{definition}\label{def:MCG}
    The \textit{mapping class group} of $S_{g}$, denote by $\text{MCG}(S_{g})$, is defined as the group of isotopy classes of diffeomorphisms.
    The \textit{mapping class group} of $S_{g,n}$, denoted by $\text{MCG}(S_{g,n})$, is the group of isotopy classes of diffeomorphisms that preserves the conjugacy class of the simple closed curve around each puncture. More precisely, if there are $n$ marked points on the closed surface $S_g$, $\text{MCG}(S_{g,n})$ is given by the isotopy classes of diffeomorphisms of $S_g$ which fix each of these marked points. Note that the isotopy also must fix each marked point. 
\end{definition}

Let $\mathcal{C}$ be the set of conjugacy classes of the simple closed curves around the punctures of $S_{g,n}$. By Dehn-Nielsen-Baer theorem \cite[Theorem 8.8]{FB}, we see that $\text{MCG}(S_{g,n})$ is isomorphic to $\text{Out}^{\ast}_{\mathcal{C}}(\pi_1(S_{g,n}))$. Hence, we have an action of $\text{MCG}(S_{g,n})$ on $\mathcal{X}(S_{g,n}, \mathcal{C})$. We will look at domain of discontinuities for this action in the next section.



\section{Primitive Stability and Simple Stability}\label{sec:SimpleStability}
Let $F_n = \langle a_1, \ldots, a_n \rangle$ i.e., a free group on $n$ letters. We define the \textit{Cayley graph} of $F_n$, denoted by $\mathcal{C}(F_n)$ as follows: vertices of this graph is given by the elements of $F_n$ and two vertices $(w_1, w_2)$ have an edge between them if $w_1^{-1} w_2 \in \{a_1, \ldots,a_n,a_1^{-1},\ldots,a_n^{-1}\}$. Observe that $\mathcal{C}(F_n)$ is a regular $2n$-valence graph with countably infinite vertices. It can also be thought of as the universal cover of bouquet of $n$ circles, $B_n$, thus identifying the fundamental group of $B_n$ to $F_n$. 

\begin{figure}[h]
	\includegraphics[width=8cm]{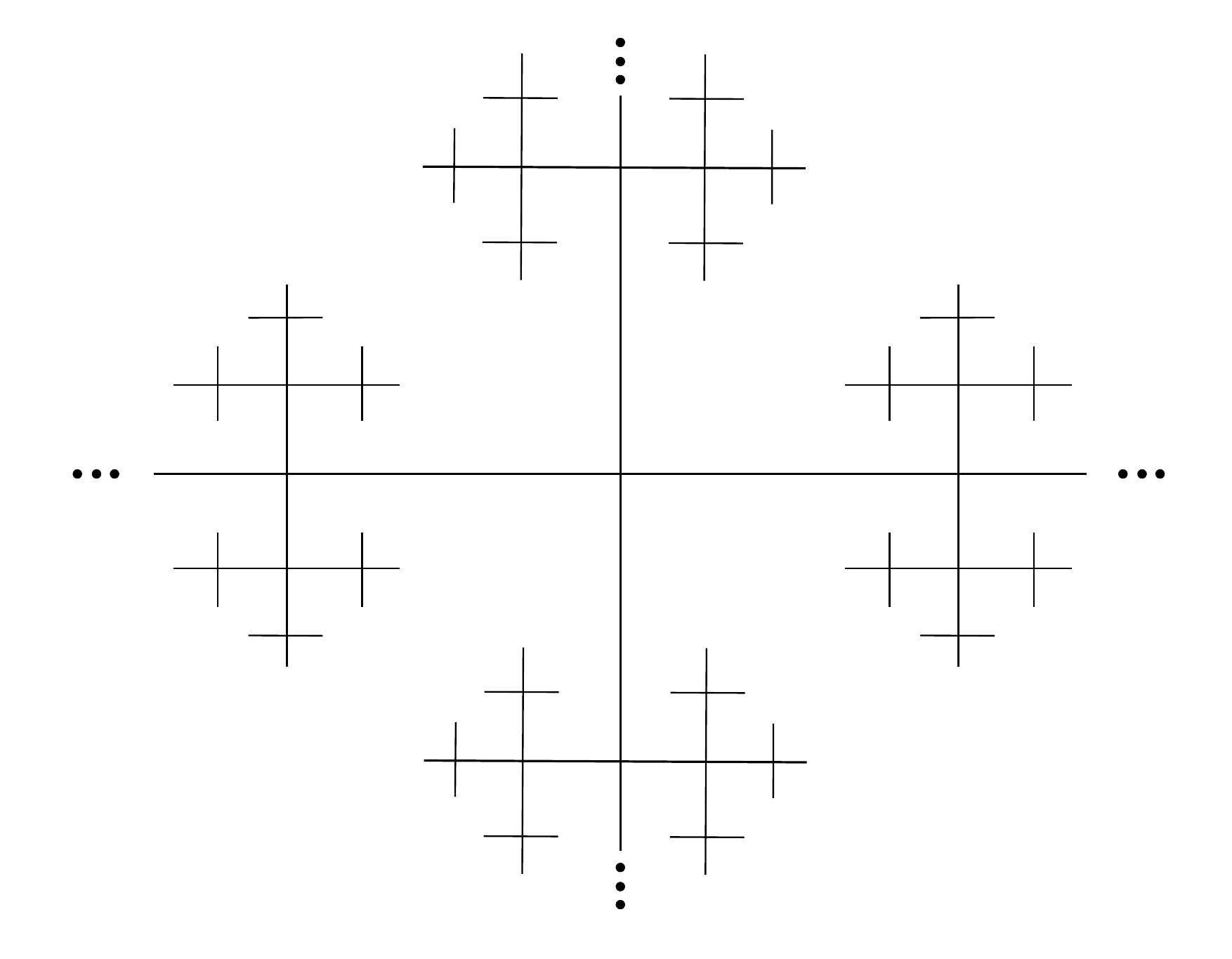}
	\caption{$\mathcal{C}(F_2)$}
	\label{fig:F2}
\end{figure}
 
Given a cyclically reduced word $w \in F_n$, thinking of it as a loop $w: (S^1,1) \rightarrow (B_n, p)$, there exists a unique lift $\widetilde{w}: \mathbb{R} \rightarrow \mathcal{C}(F_n)$ such that $\widetilde{w}(0) = e$, where $e$ denotes the empty word. Suppose $w =  w_1 \cdots w_n$, one can see that this lift is given by the path passing through $e$ that follows the infinite reduced word on  $\cdots w_1 \cdots w_n \cdots w_1 \cdots w_n \cdots$. A word $w \in F_n$ is said to be \textit{primitive} if it belongs to some free generating set of $F_n$. 

Let $\rho: F_n \rightarrow \pslr$ be a representation and $z \in \mathbb{H}^2$ be a fixed basepoint. We can define a map $\tau_{\rho,z}: \mathcal{C}(F_n) \rightarrow \mathbb{H}^2$ as follows:
\begin{itemize}
	\item $\tau_{\rho,z}(e) = z$.
	\item If $(v,w)$ is an edge in $\mathcal{C}(F_n)$, then it maps to the geodesic joining $\rho(v)(z)$ and $\rho(w)(z)$.
\end{itemize}
It is easy to note that $\tau_{\rho,z}$ as defined above is $\rho$-equivariant.

\begin{definition}
	Fix $z \in \mathbb{H}^2$. A representation $\rho:F_n \rightarrow \pslr$ is said to be \textit{primitive-stable}, if there exists $C \geq 1,\epsilon > 0$ such that for any lift $\widetilde{w}: \mathbb{R} \rightarrow \mathcal{C}(F_n)$ of a cyclically reduced primitive word $w \in F_n$, the following holds for all $t,s \in \mathbb{R}$:
	\[\dfrac{1}{C}|t-s| - \epsilon \leq d(\gamma_{\widetilde{w}}(t),\gamma_{\widetilde{w}}(s)) \leq C|t-s|+ \epsilon\] 
	where $\gamma_{\widetilde{w}} \coloneqq \tau_{\rho,z}(\widetilde{w})$. Here, note that the constants $(C, \epsilon)$ are independent of the primitive word $w$ and its lift.
\end{definition}

Schottky representations are easily seen to be primitive-stable (see \cite[Lemma 3.2]{Minsky}). In \cite{Minsky}, Minsky proves that the set of primitive-stable representations $\mathcal{PS}(F_n)$ is strictly larger that the set of Schottky representations (see \cite[Theorem 1.1]{Minsky}).
We slightly modify the definition of primitive stability to suit our context of mapping class group actions on relative character varieties, by interpreting $F_n$ as a fundamental group of a surface with punctures. In this section, we assume that all surfaces have genus at least one and at least one puncture. This implies that $\pi_1(S_{g,n})$ is free. 

\begin{definition}\label{def:SimpleStability}
	Fix $z \in \mathbb{H}^2$. A representation $\rho:\pi_1(S_{g,n}) \rightarrow \pslr$ is said to be \textit{simple-stable}, if there exists $C \geq 1,\epsilon > 0$ such that for the lift $\widetilde{w}: \mathbb{R} \rightarrow \mathcal{C}(\pi_1(S_{g,n}))$ of a cyclically reduced word representing a non-separating simple closed curve $w \in \pi_1(S_{g,n})$, the following holds for all $t,s \in \mathbb{R}$:
	\[\dfrac{1}{C}|t-s| - \epsilon \leq d(\gamma_{\widetilde{w}}(t),\gamma_{\widetilde{w}}(s)) \leq C|t-s|+ \epsilon\] 
	where $\gamma_{\widetilde{w}} \coloneqq \tau_{\rho,z}(\widetilde{w})$. Here, again the constants $(C, \epsilon)$ are independent of the simple closed curve $w$ and its lift.
\end{definition}

Denote the set of simple-stable representations of $\pi_1(S_{g,n})$ by $\mathcal{SS}(S_{g,n})$ and the set of primitive-stable representations by $\mathcal{PS}(S_{g,n})$. As every non-separating curve can be taken to another via a homeomorphism, we see that all non-separating curves are primitive elements of $\pi_1(S_{g,n})$. This implies that $\mathcal{PS}(S_{g,n}) \subset \mathcal{SS}(S_{g,n})$. 

\begin{remark}{\ } \begin{enumerate}
    \item We do not include separating curves in our definition because they may not be primitive. For example, the commutator of two elements of fundamental groups having simple closed curve representatives intersecting at one point represents a separating simple closed curve that is not primitive.  
    \item The definitions for primitive-stability and simple-stability can be extended to $\pslc$ representations by replacing $\mathbb{H}^2$ with $\mathbb{H}^3$.
\end{enumerate}
    
\end{remark}

\begin{remark}\label{rem:simpstabaltdef}
    To prove primitive stability or simple stability, it is enough to prove that there exists $C \geq 1, \epsilon>0$ such that
    \[\dfrac{1}{C}\Vert w^\prime \Vert - \epsilon \leq d(O, \rho(w^\prime) \cdot O) \leq C \Vert w^\prime \Vert+ \epsilon\] 
    for every subword $w'$ of a infinite reduced word corresponding to the lifts of primitive elements or simple closed curves. This is because these subwords form a quasi-dense subset $\gamma_{\widetilde{w}}\left(\mathbb{Z}\right)$ of $\gamma_{\widetilde{w}}\left(\mathbb{R}\right)$.
\end{remark}

      

In the following proposition, we see some first examples of simple-stable representations. They arise as holonomy of hyperbolic structures (not necessarily complete) on punctured surfaces. 

\begin{proposition}{\ }
    \begin{enumerate}
        \item The holonomy of a complete hyperbolic structure on $S_{g,n}$ with finite-area is simple-stable.
        \item The holonomy of a hyperbolic structure on $S_{g,n}$ with totally geodesic boundary is primitive-stable.
    \end{enumerate}
\end{proposition}

\begin{proof}{\ }
    \begin{enumerate}
        \item Firstly, note that all essential simple closed geodesics of the complete hyperbolic surface with finite-area lies inside a compact subset $K$ obtained by chopping off cusp-neighbourhoods. This is because of the following: if a geodesic gets sufficiently deep into the cusp-neighbourhood, its lift in the upper-half plane intersects adjacent sides of the fundamental domain with vertex at the lift of the puncture and thus, from \cite[Lemma 5.1]{BirmanSeries}, it will have a self-intersection.  This compact subset $K$ retracts onto the 1-skeleton which can be identified with the bouquet $B_{2g+n-1}$. This naturally lifts to a quasi-isometry between $\widetilde{K} \subset \mathbb{H}^2$ and $\mathcal{C}(\pi_1(S_{g,n}))$. 
        \item Let $\rho:\pi_1(S_{g,n}) \rightarrow \pslr$ be the holonomy of a hyperbolic structure with totally geodesic boundary. There exists a convex $\pi_1(S_{g,n})$-invariant subset $\mathscr{C}$ of $\mathbb{H}^2$ such that $\pi_1(S_{g,n})$ acts properly and co-compactly. This gives us a quasi-isometry between $\mathcal{C}(\pi_1(S_{g,n}))$ and $\mathscr{C}$. Hence, it takes all bi-infinite geodesics in $\mathcal{C}(\pi_1(S_{g,n}))$, in particular the lifts of primitive words, to uniform quasi-geodesics in $\mathbb{H}^2$.
    \end{enumerate}
\end{proof}
Observe that, in the proposition above, both the cases we have are of discrete representations. In the next section, we will prove the existence of an infinite family of indiscrete simple-stable representations.
\begin{theorem}
	$\mathcal{SS}(S_{g,n})$ is an open subset of $\mathcal{X}(S_{g,n})$.
\end{theorem}

\begin{proof}
	This follows directly from \cite[Lemma 3.2]{Minsky}. The proof in the paper also shows that the uniform quasi-geodesic constants vary continuously with respect to the representations.
\end{proof}

For $w \in \pi_1(S_{g,n})$, let $\Vert w \Vert$ be the word length of a cyclically reduced representative of $w$. In other words, it is the minimum over the length of all reduced words of the elements in the conjugacy class of $w$. 

\begin{lemma}\label{lem:boundedwordlength_mappingclasses}
	The set \[\left\{[\psi] \in \text{MCG}(S_{g,n}) \Big \vert \Vert \psi(w) \Vert \leq C \Vert w \Vert, \text{ for all non-separating simple closed curves } [w] \right\}\] is finite for any $C > 0$. Here, $[w]$ denotes the conjugacy class of $w$ in $\pi_1(S_{g,n})$.
\end{lemma}

\begin{proof}
        To prove this lemma, we rely on certain definitions and results from Section~\ref{sec:SBQconditions}, which is self-contained and can be understood independently.
         \begin{figure}[h]
		\includegraphics[width=.5\linewidth]{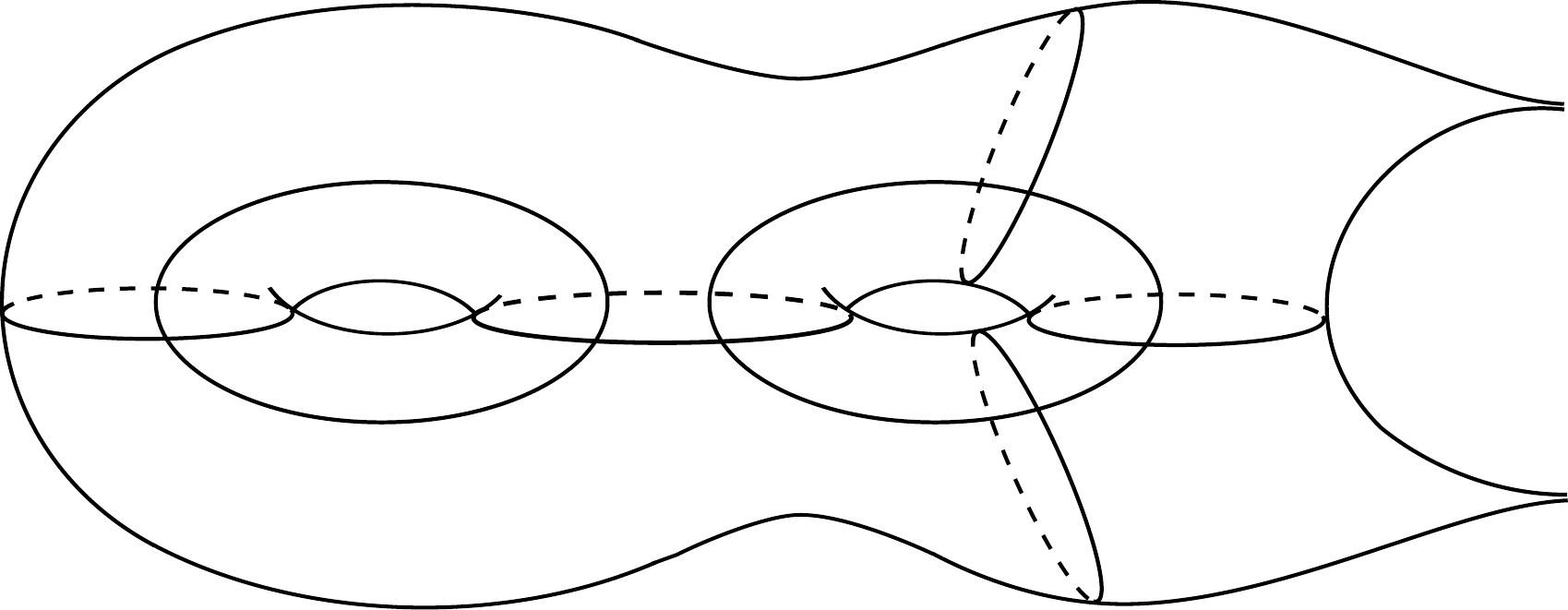}
		\captionof{figure}{A filling for genus 2 surface with two punctures}
		\label{fig:Filling_genus2}
	\end{figure}
    
        Firstly, as free homotopy classes of curves are in one-to-one correspondence with conjugacy classes in $\pi_1(S_{g,n})$, one can talk about a conjugacy class of a simple closed curve being non-separating. Thus, the set in the statement is well-defined.
        
        We use the Alexander method \cite[Proposition 2.8]{FB} to prove this proposition.  It states that a mapping class is completely determined by its action on a filling (see Definition~\ref{def:Filling}) of $S_{g,n}$ consisting of finitely many non-separating simple closed curves $\alpha_1, \cdots, \alpha_N$ such that:
        \begin{itemize}
            \item $\alpha_1, \ldots, \alpha_N$ are in general position
            \item $\alpha_i$s belong to distinct isotopy classes.
            \item given any triple of them, there exists at least a pair among the triple that are disjoint.
        \end{itemize}

        To construct such a filling, we begin with giving an admissible hyperbolic cone surface structure on $S_{g,n}$. Now, we apply Lemma~\ref{lem:AdmissibleFillingExistence} to find an admissible filling (see Definition~\ref{def:AdmissibleFilling}). From the construction of admissible filling, we observe that it satisfies the above three properties.
        
        Let $L = \displaystyle \max_{i=1}^N \Vert \alpha_i \Vert$. This is well-defined, since $\Vert \cdot \Vert$ depends only on the conjugacy class. Now, we know that $\Vert \psi(\alpha_i) \Vert \leq C \Vert \alpha_i \Vert \leq CL$. Since there are only finitely many reduced words of bounded length, this implies that there are only finitely many choices for the images of $\alpha_i$. Hence, by the Alexander method \cite[Proposition 2.8]{FB}, we see that there are only finitely many such mapping classes.
\end{proof}


We will now prove the first theorem stated in the introduction.

\ProperDiscontinuity*
\begin{proof}
	Let $B$ be a compact subset of $\mathcal{SS}(S_{g,n})$ and suppose that for $[\phi] \in \text{MCG}(S_{g,n})$, $\left([\phi]\right) \cdot B \cap B \neq \emptyset$. We need to show that there exists only finitely many such $[\phi]$. Now, we estimate $\Vert \phi(w) \Vert$, where $w$ is a cyclically reduced word representing a simple closed curve in $S_{g,n}$. 

        Let $S = \{a_1, b_1, \ldots, a_g,b_g,c_1,\ldots,c_n\}$ be the standard generating set of $\pi_1(S_{g,n})$ and $O \in \mathbb{H}^2$. First, let $K = \max\left\{d(O, \rho(v) \cdot O) \Big \vert v \in S \cup S^{-1} \right\}$. Now, suppose that $w = w_1 w_2 \ldots w_m$, where $w_i \in S \cup S^{-1}$. Then,
	\begin{align*}
		d(O, \rho(w)\cdot O) &\leq d(O, \rho(w_1) \cdot O) + d(\rho(w_1) \cdot O, \rho(w_1 w_2) \cdot O) + \cdots \\& \hspace{2em} \cdots+ d(\rho(w_1 \cdots w_{m-1}) \cdot O, \rho(w) \cdot O) \\
		&= d(O, \rho(w_1) \cdot O) + d(O, \rho(w_2) \cdot O) + \cdots + d(O, \rho(w_n) \cdot O) \\
		&\leq K + \cdots +K\\
		&= K \Vert w \Vert 
	\end{align*}
	
	Now, by compactness of $B$ and continuity, there exists a $L > 0$, such that \[\dfrac{d(O,\rho(w). O)}{\Vert w \Vert} \leq L\] for all $\rho \in B$.
	
	  By simple-stability, we get that, for $\rho$,
	\[\dfrac{d(O,\rho(w). O)}{\Vert w \Vert} \geq k_\rho > 0.\]
	
	  Again, by continuity and compactness of $B$, we get that for all $\rho \in B$, there exists a uniform constant $k > 0$ such that,
	\[\dfrac{d(O,\rho(w). O)}{\Vert w \Vert} \geq k > 0.\]
	
	Suppose $\rho \in B$ such that $[\phi] \cdot \rho \in B$, then 
	\begin{align*}
		\dfrac{d(O, \rho\left(\phi(w)\right) \cdot O)}{\Vert \phi(w) \Vert} &\geq k \\
		\implies \Vert \phi(w) \Vert &\leq \frac{1}{k} d(O, \rho\left(\phi(w)\right)) \\
		\implies \Vert \phi(w) \Vert &\leq \frac{1}{k} d(O, [\phi].\rho(w) \cdot O) \\
		\implies \Vert \phi(w) \Vert &\leq \frac{K}{k} \Vert w \Vert
 	\end{align*}
	
	There are only finitely many mapping classes that satisfy the last inequality for all essential simple closed curves, by Lemma~\ref{lem:boundedwordlength_mappingclasses}. Hence, the action is properly discontinuous.

    To prove the second part of the theorem, it is enough to prove that $\text{MCG}(S_{g,n})$ preserves $\mathcal{X}(S_{g,n}, \mathcal{C})$. As $\text{MCG}(S_{g,n})$ preserves the conjugacy classes of simple closed curve around the punctures, it preserves the set $\mathcal{C}$. Thus, $\mathcal{X}(S_{g,n}, \mathcal{C})$ is $\text{MCG}(S_{g,n})$ invariant.
	
\end{proof}

\begin{remark}
	The above theorem can be thought of as an analogue of \cite[Theorem 3.3]{Minsky}. By Dehn-Nielsen-Baer theorem \cite[Theorem 8.8]{FB}, we see that $\textup{MCG}(S_{g,n}) \subset \textup{Out}(\pi_1(S_{g,n}))$. Hence, $\textup{MCG}(S_{g,n})$ has a domain of discontinuity at least as large as that of $\textup{Out}(\pi_1(S_{g,n}))$. We shall see later that for $n \geq 2$, we have a strictly larger domain of discontinuity for $\textup{MCG}(S_{g,n})$ (see Remark~\ref{rem:MainTheorems}).
\end{remark}


\section{Holonomies of Hyperbolic Cone Surfaces}\label{sec:Holonomies}

Let $S$ be a hyperbolic cone surface and $P \subset S$ be the set of cone-points in $S$. Note that $S \setminus P$ has a hyperbolic atlas i.e., charts into $\mathbb{H}^2$ with hyperbolic isometries as transition maps. This gives us a developing map-holonomy pair, i.e., there exists a local isometry $\text{dev}: \widetilde{S\setminus P} \rightarrow \mathbb{H}^2$ and a homomorphism 
$\rho: \pi_1(S \setminus P) \rightarrow \pslr$ such that the following diagram commutes.

\[\begin{tikzcd}
	{\widetilde{S \setminus P}} && {\mathbb{H}^2} \\
	\\
	{\widetilde{S \setminus P}} && {\mathbb{H}^2}
	\arrow["{\text{dev}}", from=1-1, to=1-3]
	\arrow["{\rho(g)}", from=1-3, to=3-3]
	\arrow["g"', from=1-1, to=3-1]
	\arrow["{\text{dev}}"', from=3-1, to=3-3]
\end{tikzcd}\]
The diagram gives us the following:
$$ \text{dev} \circ g = \rho(g) \circ \text{dev} $$

Also, the developing map is unique upto post-composing by an isometry of $\mathbb{H}^2$ i.e., if $\text{dev}': \widetilde{S\setminus P} \rightarrow \pslr$ is another developing map for $S$, then $\text{dev}' = A \circ \text{dev}$, where $A \in \pslr$. The corresponding holonomies differ by conjugation by an element of $\pslr$, i.e., $\rho ' = A \cdot \rho \cdot A^{-1}$. 

\subsection{Proof of Theorem 1.2}
In the rest of this section, we provide the proofs of our main theorem which we restate below:

\ConeSurfacesSimpleStability*


We will prove some facts about cone-surfaces and their universal covers before proving the above theorems. Let $K$ be the compact set obtained after removing cone-neighbourhoods around each cone-point such that all simple closed geodesics lie in $K$. Such a neighbourhood exists because: Theorem~\ref{thm:scgeodesic_existence} gives us that simple closed geodesics do not pass through the cone-points and by Lemma~\ref{lem:cone_nbhd} there exists a neighbourhood of these cone-points that simple closed geodesics do not enter. By proof of Lemma~\ref{lem:cone_nbhd}, this compact set $K$ can be thought of as obtained by removing sectors $\mathcal{S}_i$ around the vertices of the polygon $\mathcal{P}$ corresponding to the cone-points before side identification. The lift of $K$ to the universal cover, $\widetilde{K}$, is thus given by $\left( \mathcal{P} \setminus \bigcup \mathcal{S}_i\right) \times \pi_1(S_{g,n})/ \sim$, such that $(s_j, \gamma) \sim (s_k,\gamma \cdot g_{jk})$, where $s_j, s_k$ are the geodesic sides of the truncated polygon $\mathcal{P} \setminus \bigcup \mathcal{S}_i$ and $g_{jk}$ are the elements of the fundamental group corresponding to the side identification of $s_j$ and $s_k$. This is just the construction in Section~$\ref{sec:HyperbolicConeSurfaces}$ restricted to $K$.

Let $d_{\widetilde{K}}$ be the induced path metric on $\widetilde{K} \subset \widetilde{S \setminus P}$ i.e., distance between two points is obtained by taking the infimum of lengths over the paths, lying completely inside $\widetilde{K}$, joining those two points. It is easy to note that, for all $\widetilde{x}, \widetilde{y} \in \widetilde{K}$,
\[d_{\widetilde{S \setminus P}}(\widetilde{x},\widetilde{y}) \leq d_{\widetilde{K}}(\widetilde{x},\widetilde{y}).\] Now, we prove some facts about the new metric space $\left( \widetilde{K}, d_{\widetilde{K}} \right)$.

\begin{proposition}
	$\left(\widetilde{K}, d_{\widetilde{S \setminus P}}\right)$ is a Cauchy-complete and locally compact metric space.
\end{proposition}
\begin{proof}
    Let $\pi: \left(\widetilde{K}, d_{\widetilde{S \setminus P}}\right) \rightarrow (K, d_{S \setminus P})$ be the covering map and $(\widetilde{x_n})$ be a Cauchy sequence in $\widetilde{K}$. As $\pi$ is a local isometry, $(x_n) \coloneqq (\pi(\widetilde{x_n}))$ is also a Cauchy sequence. Now, observe that $K$ is compact, in particular, it is complete. This implies that there exists $x \in K$ such that $x_n \rightarrow x$. Take an evenly covered neighbourhood $U \subset K$ of $x$, then there exists $\widetilde{U} \subset \pi^{-1}(U) \cap \widetilde{K}$ such that $(\widetilde{x_m}) \in \widetilde{U}$ for all large $m$. This implies that $\widetilde{x_m} \rightarrow \widetilde{x}$, where $\widetilde{x} \coloneqq \pi^{-1}(x) \cap \widetilde{U}$. 
	
    As every point in $\widetilde{S \setminus P}$ has a neighbourhood isometric to a hyperbolic disk, it is locally compact. Now, as $\widetilde{K}$ is a closed subspace of $\widetilde{S \setminus P}$, $\widetilde{K}$ is also locally compact.  
\end{proof}

\begin{proposition}
	The identity map $\left(\widetilde{K}, d_{\widetilde{K}} \right) \rightarrow \left(\widetilde{K}, d_{\widetilde{S \setminus P}}\right)$ is a homeomorphism.
\end{proposition}
\begin{proof}
	First note that, for any $\epsilon>0$, $B_{\epsilon}^{\widetilde{K}}(\widetilde{x}) \subset B_{\epsilon}^{\widetilde{S \setminus P}}(\widetilde{x})$. Hence, the identity map from $\left(\widetilde{K}, d_{\widetilde{K}}\right) \rightarrow \left(\widetilde{K}, d_{\widetilde{S \setminus P}}\right)$ is a continuous map. Now, define $s \coloneqq d _{\widetilde{S \setminus P}}\left(\widetilde{x}, \partial B_{\epsilon}^{\widetilde{K}}(\widetilde{x})\right)$. Then, there exists $s' \ll s$ such that $B_{s'}^{\widetilde{S \setminus P}}(\widetilde{x}) \cap \widetilde{K} \subset B_{\epsilon}^{\widetilde{K}}(\widetilde{x})$. This is because for small enough $s^\prime$, $B_{s'}^{\widetilde{S \setminus P}}(\widetilde{x}) = \text{exp}_{\widetilde{x}} (B_s^\prime(0))$, where $\text{exp}_{\widetilde{x}}$ is the Riemannian exponential map of $\widetilde{S \setminus P}$ and $B_{s^\prime}(0) \subset T_{\widetilde{x}}\left( \widetilde{S \setminus P} \right)$.
 
	Let $\delta_k, \epsilon_k \rightarrow 0$ such that $B_{\delta_k}^{\widetilde{S \setminus P}}(\widetilde{x}) \subset B_{\epsilon_k}^{\widetilde{K}}(\widetilde{x})$. Now, suppose $\widetilde{x_n} \rightarrow \widetilde{x}$ in $\left(\widetilde{K}, d_{\widetilde{S \setminus P}}\right)$, then for large $N$, $\widetilde{x_n} \in B_{\delta_k}^{\widetilde{S \setminus P}}(\widetilde{x}) \subset B_{\epsilon_k}^{\widetilde{K}}(\widetilde{x})$ for all $n \geq N$. Therefore, $\widetilde{x_n} \rightarrow \widetilde{x}$ in $\left(\widetilde{K}, d_{\widetilde{K}}\right)$. This implies that the identity map $\left(\widetilde{K}, d_{\widetilde{S \setminus P}}\right) \rightarrow \left(\widetilde{K}, d_{\widetilde{K}} \right)$ is a continuous inverse of the continuous identity map. Hence, it is a homeomorphism.
\end{proof}

\begin{lemma}
	$\left(\widetilde{K}, d_{\widetilde{K}}\right)$ is a proper geodesic metric space.
\end{lemma}
\begin{proof}
	Firstly observe that $\left(\widetilde{K}, d_{\widetilde{K}}\right)$ is locally compact, as it is homeomorphic to $(\widetilde{K}, d_{\widetilde{S \setminus P}})$. Now, suppose $(\widetilde{x}_n)$ is a Cauchy sequence in $\left(\widetilde{K}, d_{\widetilde{K}}\right)$, then $\left(\widetilde{x}_n\right)$ is also Cauchy in $(\widetilde{K}, d_{\widetilde{S \setminus P}})$. From completeness, we get that there exists an $\widetilde{x} \in \widetilde{K}$ such that $\widetilde{x}_n \rightarrow \widetilde{x}$ in $\left(\widetilde{K}, d_{\widetilde{S \setminus P}}\right)$. Now, as identity map is a homeomorphism, we get that $\widetilde{x}_n \rightarrow \widetilde{x}$ in $\left(\widetilde{K}, d_{\widetilde{K}}\right)$. Hence, $(\widetilde{K},d_{\widetilde{K}})$ is locally compact and complete, thus it is proper (\cite[Corollary 3.8]{BH}). 	
\end{proof}

\begin{lemma}\label{lem:QuasiIsometry}
	$\mathcal{C}(\pi_1(S \setminus P))$ is quasi-isometric to $\left(\widetilde{K},d_{\widetilde{K}}\right)$.
\end{lemma}
\begin{proof}
	$\mathcal{C}(\pi_1(S \setminus P))$ acts on $\left(\widetilde{K}, d_{\widetilde{K}}\right)$ via isometries, since lengths of curves are preserved under both induced path metrics and isometries. The action is also properly discontinuous, since it is the action by deck transformations. In summary, $\mathcal{C}(\pi_1(S \setminus P))$ acts properly and co-compactly on $\left(\widetilde{K}, d_{\widetilde{K}}\right)$ and thus $\mathcal{C}(\pi_1(S \setminus P))$ is quasi-isometric to $\left(\widetilde{K}, d_{\widetilde{K}}\right)$.
\end{proof}

\begin{remark}
    Note that Lemma~\ref{lem:QuasiIsometry} is not true for $\left(\widetilde{K},d_{\widetilde{S \setminus P}}\right)$. This is because $\left(\widetilde{K},d_{\widetilde{S \setminus P}}\right)$ is not a geodesic metric space as the geodesic between two points sufficently near the boundary would lie outside $\widetilde{K}$. This is resolved by taking the induced metric because we consider only paths lying completely inside $\widetilde{K}$.
\end{remark}

 \begin{proof}[Proof of Theorem 1.2]
    Let $\widetilde{x} \in \widetilde{S \setminus P}$ and $w$ be a subword of an infinite reduced word representing a lift of the simple closed curve, say $\gamma$ in $\mathcal{C}(\pi_1(S_{g,n}))$. We prove that there exists a geodesic joining $\widetilde{x}$ and $w \cdot \widetilde{x}$ in $\widetilde{S \setminus P}$. Assuming the contrary would imply that the minimum distance between $\widetilde{x}$ and $w \cdot \widetilde{x}$ is realised by a curve that passes through a lift of a cone-point. There exists a subarc of this distance realising curve lying completely inside the neighbourhood of the lift of the cone-point. Let the homotopy class of the image of this subarc in the cone-neighbourhood on the surface be $\mathcal{H}_k$. As this geodesic arc passes through the cone-point,  by Proposition~\ref{thm:conearcs}, we see that $k \geq \lfloor \dfrac{\pi}{\theta} \rfloor$, where $\theta$ is the cone-angle. But, since $\theta < \pi$, we get that $k > 1$. From the discussion regarding arcs on hyperbolic cones in Section~\ref{sec:HyperbolicConeSurfaces}, any lift of the homotopy class $\mathcal{H}_k$ must pass through $k+1$ sectors in the lift of the cone. Hence, the subarc under consideration must pass through at least three sectors in the lift of the cone. This implies that the image of $w$ on the surface   goes around the cone-point at least twice, which in turn implies that $w$ contains $c^2$ as a subword, where $c$ is the generator corresponding to the peripheral curve around the cone-point. This would imply that the cyclically reduced word representing the simple closed curve has $c^2$ as a subword. This means that the simple closed curve $\gamma$ goes around the cone-point twice, which is impossible by Proposition~\ref{prop:selfintersection}. Hence, there exists a geodesic joining $\widetilde{x}$ and $w \cdot \widetilde{x}$, for every subword of an infinite reduced word representing a lift of a simple closed curve.

    Fix an $\widetilde{x} \in \widetilde{S \setminus P}$, we find a subspace $\mathcal{K}$ of $\widetilde{S \setminus P}$, such that $\mathcal{K}$ contains $\widetilde{K}$ and is $\pi_1(S_{g,n})$-invariant and $\pi(\mathcal{K})$ is compact. Also, $\mathcal{K}$ has the property that the geodesic joining $\widetilde{x}$ and $w \cdot \widetilde{x}$ lies completely inside $\mathcal{K}$, for every subword $w$ of a lift of simple closed curve. Such a subspace is guaranteed by Lemma~\ref{lem:cone_nbhd} as we now explain. As $\widetilde{x}$ lies a definite distance away from the cone-points, any geodesic starting at $\widetilde{x}$ enters any cone neighbourhood from a distance $d_{\widetilde{x}} > 0$. From Lemma~\ref{lem:cone_nbhd}, there exists a neighbourhood $U_{d_{\widetilde{x}}}$ of the cone-point that the geodesic does not enter. Note that these neighbourhoods depend only on $\widetilde{x}$ and they shrink as $\widetilde{x}$ goes closer to the cone-point.
    
    Now, by Lemma~\ref{lem:QuasiIsometry}, $(\mathcal{K},d_{\mathcal{K}})$ is a proper geodesic space and it is quasi-isometric to $\mathcal{C}(\pi_1(S \setminus P))$. Therefore, there exists $C \geq 1, \epsilon > 0$ such that for every $w \in \pi_1(S_{g,n})$
    \[ \frac{1}{C} \Vert w \Vert - \epsilon \leq d_{\mathcal{K}}(\widetilde{x}, w \cdot \widetilde{x}) \leq C \Vert w \Vert + \epsilon\]

    Suppose $w$ is a subword of an infinite reduced word representing a lift of a simple closed curve. There exists a geodesic in $\mathcal{K}$ which is also a geodesic in $\widetilde{S \setminus P}$ joining $\widetilde{x}$ and $w \cdot \widetilde{x}$. Now the image of this geodesic under the developing map $\text{dev}: \widetilde{S \setminus P} \rightarrow \mathbb{H}^2$ is again a geodesic, since $\text{dev}$ is a local isometry. Now assuming $\text{dev}(\widetilde{x}) = O \in \mathbb{H}^2$, this implies that 
    \[d_{\mathcal{K}}(\widetilde{x}, w \cdot \widetilde{x}) = d_{\mathbb{H}^2}(O, \rho(w) \cdot O)\]

    Hence, for every cyclically reduced subword $w$ of a infinite reduced word representing a lift of a simple closed curve,
    \[ \frac{1}{C} \Vert w \Vert - \epsilon \leq d_{\mathbb{H}^2}(O, \rho(w) \cdot O) \leq C \Vert w \Vert + \epsilon\]

    Hence, by Remark~\ref{rem:simpstabaltdef}, this implies that, $\rho$ is simple-stable.

    The proof of the second part proceeds exactly similar to the previous one. The only difference is that we need to find $\mathcal{K} \subset \widetilde{S \setminus P}$ such that all geodesics joining $\widetilde{x}$ and $w \cdot \widetilde{x}$ is contained in $\mathcal{K}$, for any subword $w$ of a cyclically reduced primitive word. There exists such a $\mathcal{K}$, because \cite[Lemma 4.4]{Minsky} implies that $c^2$, where $c=\prod_{i=1}^g [a_i, b_i]$, cannot be a subword of a primitive word, i.e., $c^2$ is blocking.  As in the proof above, if there is no geodesic joining $\widetilde{x}$ and $w \cdot \widetilde{x}$, then the shortest curve passes through the cone-point implying that the curve $w$ goes around the cone-point at least twice implying that $c^2$ is a subword of $w$, which is a contradiction. The rest of the proof proceeds exactly as above. Hence, $\rho$ is primitive-stable.  
\end{proof}

\begin{remark} \label{rem:MainTheorems}From the theorem above, we make following observations:
    \begin{enumerate}
        \item The holonomy representation of a cone surface with cone-points is indiscrete, whenever one of the cone-angles is an irrational multiple of $\pi$. Hence, Theorem~\ref{thm:main1} gives examples of indiscrete simple-stable and indiscrete primitive-stable representations.
        \item In hyperbolic cone surfaces with more than one cone-point, the simple closed curves around the cone-point are primitive and under the holonomy they map to elliptic elements. Hence, they cannot be primitive-stable. These holonomies form the examples of simple-stable representations that are not primitive-stable. 
    \end{enumerate}    
\end{remark}

\subsection{Strong Simple Stability}\label{StrongSimpleStability}
One can define a slightly stronger notion of simple-stability by dropping the requirement of simple closed curves to be non-separating. Call this \textit{strong simple-stability}. Although every simple closed curve in a punctured sphere is separating, we can talk about strongly simple-stable representations of $\pi_1(S_{0,n})$. The following fact is straightforward from the proof of Theorem~\ref{thm:main1}, as the argument there depended only the simplicity of the closed curves:

\begin{proposition} \label{prop:StrongSimpleStabilityAdmissibleSurface}
    The holonomy of an admissible cone surface is strongly simple-stable.
\end{proposition}

Now, note that the fundamental group of a punctured sphere $S_{0,n}$ is given by:
\[\pi_1(S_{0,n}) \coloneqq \langle c_1, \ldots, c_n \vert c_1 \cdots c_n \rangle\]

We have an equivalent version of Proposition~\ref{prop:admissiblePolygon} for admissible punctured spheres, i.e., hyperbolic cone spheres with cone-angles less than or equal to $\pi$.

\begin{proposition}\label{prop:PolygonForConeSpheres}
    Any admissible hyperbolic cone sphere $S$ is obtained by identifying sides of a convex hyperbolic polygon by isometries.
\end{proposition}
\begin{proof}
    Let the admissible cone sphere have $n$ cone-points, say $p_1, \ldots, p_n$. Let $g_k$ be the geodesic joining $p_1$ to $p_k$. Note that $g_i$ and $g_j$ are disjoint everywhere except at $p_1$, whenever $i \neq j$, since there cannot be bigons. Now, at $p_1$, starting at $g_2$ we have a ordering on the geodesics coming out of $p_1$. Without loss of generality, assume that the ordered geodesics are $g_2,g_3, \ldots, g_n$. Now, let $m_k$ be the geodesic joining $p_k$ and $p_{k+1}$ for $1 \leq k \leq n-1$ and $m_n$ be the geodesic joining $p_n$ and $p_1$. Again, $m_i$ and $m_j$ are disjoint, whenever $|i-j| \geq 2 \text{ mod }n$ and in other cases, they intersect at cone-points.

    \begin{figure}[h]
		\labellist
            \tiny
		\pinlabel $p_1$ at 180 215
            \pinlabel $p_5$ at 90 150
            \pinlabel $p_2$ at 260 150
            \pinlabel $p_3$ at 260 320
            \pinlabel $p_4$ at 90  310

            \pinlabel $g_5$ at 140 220
            \pinlabel $g_2$ at 210 220
            \pinlabel $g_3$ at 205 250
            \pinlabel $g_4$ at 140 250

            \pinlabel $m_5$ at 180 180
            \pinlabel $m_2$ at 240 240
            \pinlabel $m_3$ at 170 295
            \pinlabel $m_4$ at 100 240

            \pinlabel $m_1$ at 600 190
            \pinlabel $m_2$ at 635 240
            \pinlabel $m_3$ at 560 295
            \pinlabel $m_4$ at 495 240
            \pinlabel $m_5$ at 545 190

            \pinlabel $p_1$ at 570 215
            \pinlabel $p_5$ at 485 150
            \pinlabel $p_2$ at 650 150
            \pinlabel $p_3$ at 650 325
            \pinlabel $p_4$ at 480 310
            
		\endlabellist
		\centering
		\centering  
		\includegraphics[width=0.7\linewidth]{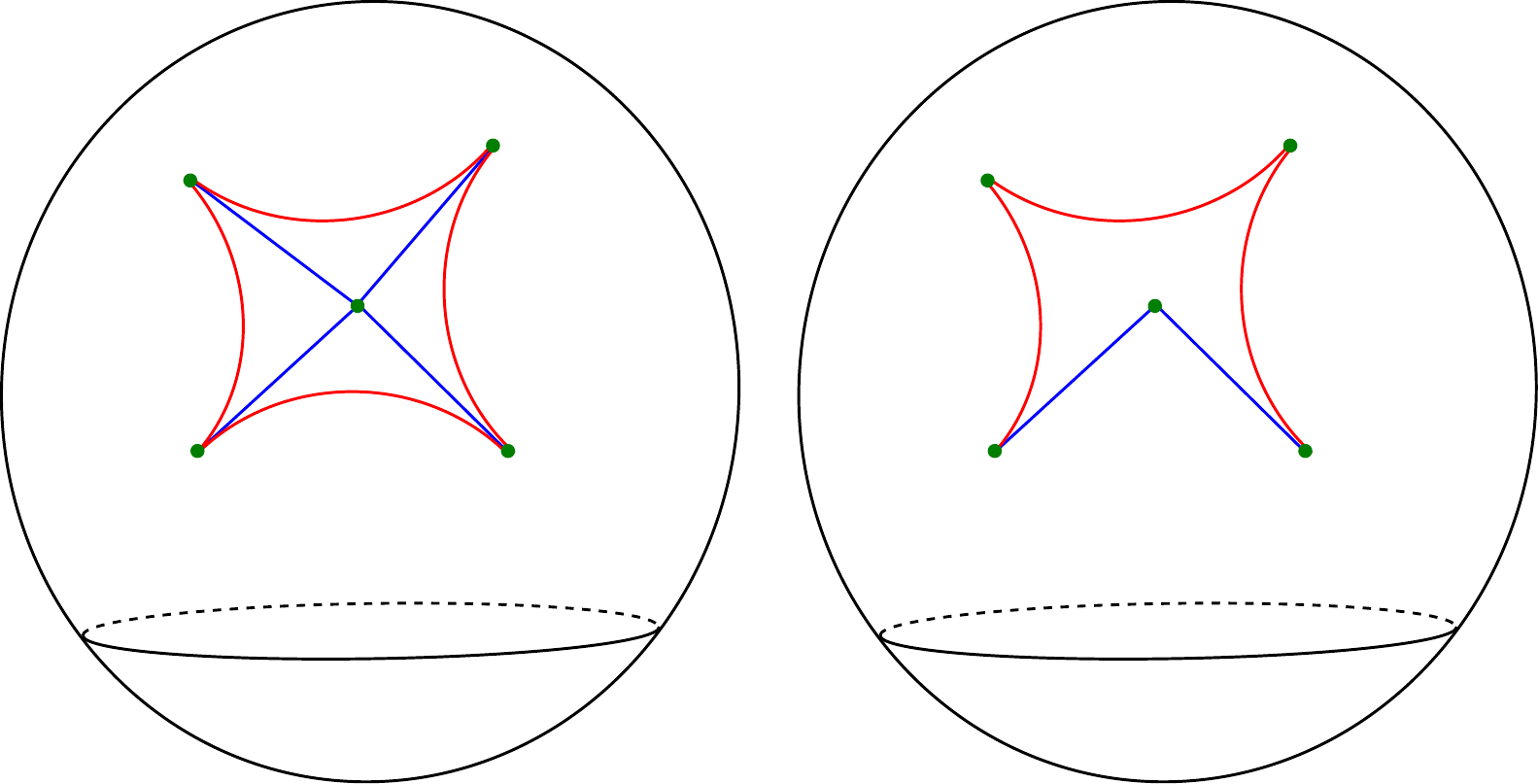}
		\captionof{figure}{Finding an admissible polygon}
		\label{fig:hypconesphere}
	\end{figure}
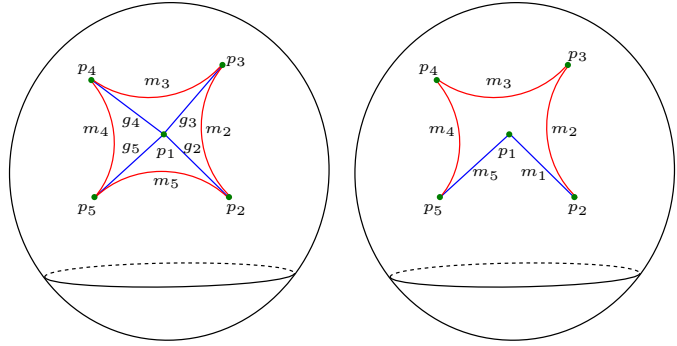

    The geodesics $m_1, \ldots, m_n$ bound a convex hyperbolic polygon $\mathcal{P}^\prime$ in the surface $S$. The convexity follows from the fact that the cone-angles are less than $\pi$. Let $x$ be a point in the interior of this polygon and $l_1, \ldots, l_n$ be the geodesics joining $x$ to the vertices. As $\mathcal{P}^\prime$ is isometrically embedded in $S$, $l_1, \ldots, l_n$ are also geodesics in $S$. Now, cutting $S$ along $l_1,\ldots,l_n$, we get a convex hyperbolic polygon $\mathcal{P}$ such that identifying the sides of $\mathcal{P}$ via isometries $c_1, \ldots, c_n$ such that $c_1 \cdots c_n = 1$ gives us $S$.
\end{proof}

Now, we prove one of the main results stated in the introduction.
\PuncturedSpheresStrongSimpleStability*    

\begin{proof}
    The proof follows exactly as that of Theorem~\ref{thm:main1} with the only difference being the hyperbolic geodesic polygon used in constructing the universal cover of given admissible punctured cone sphere. Specifically, we construct the universal cover, using the hyperbolic geodesic polygon $\mathcal{P}$ obtained from Proposition~\ref{prop:PolygonForConeSpheres}, as described in Remark~\ref{rem:UniversalCover}. 
\end{proof}

\begin{remark}
    
\end{remark}

\section{Simple Bowditch Q-Conditions}\label{sec:SBQconditions}
First we recall the definition of Bowditch Q-conditions and then define a slightly modified version for our purpose.
\begin{definition}[BQ-Conditions]\label{def:BQ}
    A representation in $\rho:F_n \rightarrow \pslr$ is said to satisfy \textit{Bowditch Q-conditions} if:
    \begin{enumerate}
        \item $\rho(w)$ is hyperbolic for every primitive word $w \in F_n$.
        \item There are only finitely many $w \in F_2$ such that $\vert \tra \left(\rho(w) \right) \vert \leq 2$. 
    \end{enumerate}
\end{definition}

We do the natural modification to this definition to obtain the following:

\begin{definition}[SBQ-Conditions] \label{def:SBQ}
     A representation in $\rho:\pi_1(S_{g,n}) \rightarrow \pslr$ is said to satisfy \textit{simple Bowditch Q-conditions} if:
    \begin{enumerate}
        \item $\rho(\gamma)$ is hyperbolic for every simple closed curve $\gamma$.
        \item There are only finitely many simple closed curves $\gamma$ such that $\vert \tra \left( \rho(\gamma) \right) \vert \leq 2$. 
    \end{enumerate}
\end{definition}

We call these $BQ$ and $SBQ$-conditions for brevity. It follows from \cite[Proposition 2.9]{LupiThesis} that primitive-stable representations satisfy $BQ$-conditions. The proof of \cite[Proposition 2.9]{LupiThesis} would go through without any modifications to prove that strongly simple-stable representations satisfy $SBQ$-conditions.

In this Section, we give a proof of the $SBQ$-conditions for the holonomy of admissible cone surfaces without the assumption of strong simple-stability.

\begin{definition}
	Let $a,b$ be distinct isotopy classes of closed curves. The intersection number of $a$ and $b$ denoted by $i(a,b)$ is defined as the infimum of the number of intersections over all representatives of $a$ and $b$ i.e.,
	$$i(a,b) = \inf_{\gamma \in a, \sigma \in b} |\gamma \cap \sigma|$$ 
\end{definition}
\begin{definition} \label{def:Filling}
		Let $\mathcal{F} = \{a_1, \ldots, a_n\}$ be a set of distinct isotopy classes of curves on $S$. Then, $\mathcal{F}$ is called a \textit{filling} of $S$, if there exists curves $\widetilde{\mathcal{F}} = \{\gamma_1, \ldots, \gamma_n\}$ in minimal position, where $\gamma_i$ is a representative for the class $a_i$, such that, $S \setminus \widetilde{\mathcal{F}}$ is a disjoint union of disks or once-punctured disks. 
\end{definition} 

Let $\mathcal{F} =  \{a_1, \ldots, a_n\} $ be a filling of $\Sigma$ and $\gamma$ is a closed curve on $S$, then define
$$i(\gamma, \mathcal{F}) = \max_{i = 1,\ldots, n} i(\gamma, a_i) $$

\begin{lemma} \label{lem:FillingLemma}
	Let $\mathcal{F}$ be a filling of $S$, possibly with boundary and punctures, consisting of non-peripheral essential simple closed curves. Define, for fixed $N \in \mathbb{N}$
	$$ \mathcal{I}_N = \{a ~ | ~ i(a,\mathcal{F}) \leq N, a \text{ is an isotopy class of a simple closed curve} \}. $$
	Then, $\mathcal{I}_N$ is a finite set for every $N \in \mathbb{N}$.
\end{lemma}

\begin{proof}
	
	Give $S$ a complete hyperbolic structure and we denote the hyperbolic surface thus obtained also by $S$. If $S$ has cusps, then there are neigbourhoods around cusps such that every simple closed geodesic is disjoint from these neighbourhoods.

	Now, consider the geodesic representatives of curves in $\mathcal{F}$. Removing cusp neighbourhoods if any, we can assume that the surface is compact, possibly with non-geodesic boundary. 
	
	Now, $S \setminus \mathcal{F}$ is disjoint union of geodesic polygons or a geodesic polygons with an open disk removed. Let $a \in \mathcal{I}_N$, then the number of arcs corresponding to $a$ in each of these polygons is bounded above. This implies that the length of $a$ is also bounded above. We know that there are only finitely many simple closed curves with lengths bounded above by any given positive real number i.e., simple length spectrum of hyperbolic surface is discrete. Hence, $\mathcal{I}_N$ is a finite set. 	
\end{proof}

\begin{definition} \label{def:AdmissibleFilling}
	A filling $\mathcal{F}$ of a hyperbolic cone surface of genus greater than or equal to 1 is said to be \textit{admissible} if it consists only of non cone-point peripheral, essential simple closed geodesics.
\end{definition}

\begin{definition} \label{def:partition}
	A \textit{partition} of a cone-surface $S$ is a set of simple closed geodesics $\mathcal{P} = \{\gamma_1, \ldots, \gamma_n\}$ such that  $ S \setminus \mathcal{P}$ is a union of hyperbolic pair of pants with three geodesic boundaries and hyperbolic annuli with one cone-point.
\end{definition}

\begin{lemma} \label{lem:AdmissibleFillingExistence}
	Every admissible hyperbolic cone surface of genus greater than or equal to one has an admissible filling.
\end{lemma}

\begin{proof}
	Let $S$ be an admissible hyperbolic cone surface with genus greater than or equal to 1 and $\gamma$ be a non cone-point peripheral simple closed curve. As $S$ is admissible, we can assume that $\gamma$ is a simple closed geodesic. We construct a partition $\mathcal{P}$ containing $\gamma$ inductively. Assume that we have found a set of simple closed geodesics $\mathcal{P} = \{\gamma, \gamma_1, \ldots, \gamma_k\}$ that are non cone-point peripheral and disjoint. 

    Let $S_1, \ldots, S_k$ be the connected components of $S \setminus \mathcal{P}_k$. Now, if any $S_i$ has genus greater than 1, then it has a non cone-point peripheral essential simple closed geodesic. If any $S_i$ has genus 0, then it has to have at least two boundary geodesics and one cone-point. Otherwise, it would imply that one of the $\gamma_i$'s is cone-point peripheral. Therefore, any of these $S_i$ either:
    \begin{enumerate}
        \item contains a non cone-point peripheral essential simple closed geodesic or
        \item is a hyperbolic pair of pants with geodesic boundaries. 
        \item is a hyperbolic annulus with one cone-point
    \end{enumerate}

    We extend the list $\mathcal{P}$ by including, for each $S_i$, one non cone-point peripheral essential simple closed geodesic, if such a geodesic exists. We continue this process until we are left with only annuli and pants. And, $\mathcal{P}$ thus obtained would be our partition.
 
    Any two boundary components of the annuli and pair of pants obtained by the partition can be connected by a simple arc. In the case of a pair of pants, we get three disjoint simple arcs and for an annulus we get one simple arc not passing through cone-point. Now, we join these arcs in such a way that they form simple closed curves of $S$:
    \begin{itemize}
        \item If $P_1$ and $P_2$ are pair of pants that share two boundary components, let $l_1, l_2$ be the arcs joining these shared boundary components. Then, concatenating $l_1$, $l_2$ and some arcs on the boundary component we get a simple closed curve intersecting these two boundary curves.
        \item If $P_1$ and $P_2$ are pair of pants that share only one boundary component, there exists a sequence of distinct pair of pants $P_1, P_2, P_3, \ldots P_n $ in $\mathcal{P}$ such that consecutive pair of pants share at least one boundary component and $P_n$ shares one boundary component with $P_1$. Now concatenate the arcs joining these boundary components and some arcs on the boundary component to form a simple closed curve.   
    \end{itemize}
        
    We construct simple closed curves as described above making sure that we use at least two arcs from every pair of pants and one arc from every annulus in the decomposition. Therefore, taking $\mathcal{P}$ along with the geodesic representatives of these newly constructed simple closed curves will give us an admissible filling. The simple closed geodesics formed this way would not be cone-point peripheral because they essentially intersect at least one of the $\gamma_i$ exactly once.
\end{proof}

\SBQconditions*

\begin{proof}
	Denote by $\mathcal{C}_S$, the set of simple closed geodesics on $S$.
	Consider the following set $\mathcal{A}_L$ defined as
	$$\mathcal{A}_L \coloneqq \{c \in \mathcal{C}_S ~ |~ l(c) < L \}$$
	where $l(c)$ denotes the length of the geodesic $c$. Let $\mathcal{F} = \{a_1, \ldots, a_n\}$ be an admissible filling. Note that by Collar Lemma of hyperbolic cone surfaces as proved in \cite{PARLIER}, every curve in this filling has a collar around them. This implies that there exists $N \in \mathbb{N}$, such that for any geodesic with $l(c) < L$, the intersection with the filling, $i(c, \mathcal{F}) < N$. Thus, by Lemma~\ref{lem:FillingLemma}, the set $\mathcal{A}$ is finite.

        We know that, length of the geodesic is given by the translation length of the corresponding hyperbolic element obtained via the holonomy. As the translation lengths are completely determined by the trace of this hyperbolic element, we see that the holonomy satisfies $SBQ$-conditions.
\end{proof}

\begin{remark}{\ }
    \begin{enumerate}
        \item Note that the holonomy of admissible cone surfaces with at least one irrational cone-angle is indiscrete. Hence, these representations form examples of indiscrete representations that satisfy $SBQ$-conditions and in particular, possess a discrete length spectrum.
        \item It will be interesting to find the examples of representations satisfying $SBQ$-conditions that are not simple-stable or determine whether such examples exist as posed in Question~\ref{Question2}.
    \end{enumerate}    
\end{remark}






\bibliographystyle{amsalpha}
\bibliography{references.bib}

\providecommand{\bysame}{\leavevmode\hbox to3em{\hrulefill}\thinspace}
\providecommand{\MR}{\relax\ifhmode\unskip\space\fi MR }
\providecommand{\MRhref}[2]{%
  \href{http://www.ams.org/mathscinet-getitem?mr=#1}{#2}
}
\providecommand{\href}[2]{#2}
\begin{thebibliography}{DKPT22}

\bibitem[BH99]{BH}
Martin~R. Bridson and Andr\'{e} Haefliger, \emph{Metric spaces of non-positive
  curvature}, Grundlehren der mathematischen Wissenschaften [Fundamental
  Principles of Mathematical Sciences], vol. 319, Springer-Verlag, Berlin,
  1999. \MR{1744486}

\bibitem[Bow98]{Bowditch1}
B.~H. Bowditch, \emph{Markoff triples and quasi-{F}uchsian groups}, Proc.
  London Math. Soc. (3) \textbf{77} (1998), no.~3, 697--736. \MR{1643429}

\bibitem[BS85]{BirmanSeries}
Joan~S. Birman and Caroline Series, \emph{Geodesics with bounded intersection
  number on surfaces are sparsely distributed}, Topology \textbf{24} (1985),
  no.~2, 217--225. \MR{793185}

\bibitem[CS83]{CULLER-SHALEN}
Marc Culler and Peter~B. Shalen, \emph{Varieties of group representations and
  splittings of {$3$}-manifolds}, Ann. of Math. (2) \textbf{117} (1983), no.~1,
  109--146. \MR{683804}

\bibitem[DKPT22]{DESPRE}
Vincent Despré, Benedikt Kolbe, Hugo Parlier, and Monique Teillaud,
  \emph{Computing a {D}irichlet domain for a hyperbolic surface},
  \url{https://arxiv.org/abs/2212.01934}, 2022.

\bibitem[DP07]{PARLIER}
Emily~B. Dryden and Hugo Parlier, \emph{Collars and partitions of hyperbolic
  cone-surfaces}, Geom. Dedicata \textbf{127} (2007), 139--149. \MR{2338522}

\bibitem[DT19]{DT}
Bertrand Deroin and Nicolas Tholozan, \emph{Supra-maximal representations from
  fundamental groups of punctured spheres to {$\rm PSL(2,\Bbb R)$}}, Ann. Sci.
  \'{E}c. Norm. Sup\'{e}r. (4) \textbf{52} (2019), no.~5, 1305--1329.
  \MR{4057784}

\bibitem[FM12]{FB}
Benson Farb and Dan Margalit, \emph{A {P}rimer on {M}apping {C}lass {G}roups},
  Princeton Mathematical Series, vol.~49, Princeton University Press,
  Princeton, NJ, 2012. \MR{2850125}

\bibitem[Gol88]{GOLDMAN2}
William~M. Goldman, \emph{Topological components of spaces of representations},
  Invent. Math. \textbf{93} (1988), no.~3, 557--607. \MR{952283}

\bibitem[Gol03]{GOLDMAN}
\bysame, \emph{The modular group action on real {${\rm SL}(2)$}-characters of a
  one-holed torus}, Geom. Topol. \textbf{7} (2003), 443--486. \MR{2026539}

\bibitem[Gol06]{GOLDMANConj}
\bysame, \emph{Mapping class group dynamics on surface group representations},
  Problems on mapping class groups and related topics, Proc. Sympos. Pure
  Math., vol.~74, Amer. Math. Soc., Providence, RI, 2006, pp.~189--214.
  \MR{2264541}

\bibitem[LF23]{TLEFILS}
Thomas Le~Fils, \emph{Holonomy of complex projective structures on surfaces
  with prescribed branch data}, J. Topol. \textbf{16} (2023), no.~1, 430--487.
  \MR{4575871}

\bibitem[Lup15]{LupiThesis}
Damiano Lupi, \emph{Primitive stability and {B}owditch conditions for rank 2
  free group representations}, Ph.D. thesis, University of Warwick, September
  2015, \url{https://wrap.warwick.ac.uk/78992/}.

\bibitem[LX20]{Lee-Xu}
Jaejeong Lee and Binbin Xu, \emph{Bowditch's {Q}-conditions and {M}insky's
  primitive stability}, Trans. Amer. Math. Soc. \textbf{373} (2020), no.~2,
  1265--1305. \MR{4068264}

\bibitem[Mar]{Maret}
Arnaud Maret, \emph{A note on {C}haracter {V}arieties},
  \url{https://arnaudmaret.com/files/character-varieties.pdf}.

\bibitem[Mar22]{BM}
Bruno Martelli, \emph{An {I}ntroduction to {G}eometric {T}opology}, 2022,
  \url{https://arxiv.org/abs/1610.02592}.

\bibitem[Mat12]{DanielMathews1}
Daniel~V. Mathews, \emph{Hyperbolic cone-manifold structures with prescribed
  holonomy {II}: higher genus}, Geom. Dedicata \textbf{160} (2012), 15--45.
  \MR{2970041}

\bibitem[Min13]{Minsky}
Yair~N. Minsky, \emph{On dynamics of {$Out(F_n)$} on {$\text{PSL}_2({\Bbb C})$}
  characters}, Israel J. Math. \textbf{193} (2013), no.~1, 47--70. \MR{3038545}

\bibitem[MPT15]{Maloni}
Sara Maloni, Fr\'ed\'eric Palesi, and Ser~Peow Tan, \emph{On the character
  variety of the four-holed sphere}, Groups Geom. Dyn. \textbf{9} (2015),
  no.~3, 737--782. \MR{3420542}

\bibitem[MW16]{Marche-Wolff}
Julien March\'e and Maxime Wolff, \emph{The modular action on
  {$\text{PSL}_2(\Bbb{R})$}-characters in genus 2}, Duke Math. J. \textbf{165}
  (2016), no.~2, 371--412. \MR{3457677}

\bibitem[Pan17]{HuipingPan}
Huiping Pan, \emph{On finite marked length spectral rigidity of hyperbolic cone
  surfaces and the {T}hurston metric}, Geom. Dedicata \textbf{191} (2017),
  53--83. \MR{3719075}

\bibitem[Ser19]{CarolineSeriesPSBQ}
Caroline Series, \emph{Primitive stability and {B}owditch's {BQ}-conditions are
  equivalent}, 2019, \url{https://arxiv.org/abs/1901.01396}.

\bibitem[TWZ06]{SPTAN}
Ser~Peow Tan, Yan~Loi Wong, and Ying Zhang, \emph{Generalizations of
  {M}c{S}hane's identity to hyperbolic cone-surfaces}, J. Differential Geom.
  \textbf{72} (2006), no.~1, 73--112. \MR{2215456}

\bibitem[TWZ08]{TWZ1}
\bysame, \emph{Generalized {M}arkoff maps and {M}c{S}hane's identity}, Adv.
  Math. \textbf{217} (2008), no.~2, 761--813. \MR{2370281}

\end{thebibliography}

\end{document}